\documentclass[hidelinks, 11pt]{amsart}

\usepackage{geometry}
\usepackage{various}
\usepackage{listings}
\usepackage{comment}
\usepackage{hyperref}
\mathtoolsset{showonlyrefs}

\begin{document}
\title{\vspace{-1cm}Squarefree numbers in short intervals}
\author{Mayank Pandey\vspace{-1cm}}
\address{Department of Mathematics, Princeton University, Princeton, NJ 08540, USA}
\email{mayankpandey9973@gmail.com}

\maketitle


\begin{abstract}
  We show that there exists $\eta > 0$ such that the interval $[X, X + X^{\frac 15 - \eta}]$ contains
  a squarefree number for all large $X$. This improves on an earlier result of Filaseta and Trifonov
  who showed that there is a squarefree number in $[X, X + cX^{\frac 15}\log X]$ for some $c > 0$
  and all large $X$.

  We introduce a new technique to count lattice points near curves satisfying further restrictions,
  which we use to bound in critical ranges the number of integers in a short interval divisible by a large square.
  This uses as an input Green and Tao's quantitative version of Leibman's theorem on the
  equidistribution of polynomial orbits in nilmanifolds.
\end{abstract}
\renewcommand{\baselinestretch}{0.74}\normalsize
\tableofcontents
\renewcommand{\baselinestretch}{1.0}\normalsize

\newpage
\section{Introduction}

A number is said to be squarefree if it is not divisible by the square of a prime.
It is classical that the set of squarefree numbers has asymptotic density $\frac{6}{\pi^2}$,
and a basic question in analytic number theory, like with the primes, is to understand
the behavior of squarefree numbers in short intervals.
One of the most basic questions about this is how large gaps between squarefree numbers
can be at large scales.

Letting $q_n$ denote the $n$th squarefree number, one might conjecture with a basic probabilistic
model that for all $n$, we have $q_{n + 1} - q_n\ll \log n$.
Current bounds are very far from such a result, 
so we introduce now some notation for the subsequent discussion. Let
\[
  \theta^* := \inf\set{\theta : \limsup_{n\to\infty} \frac{q_{n + 1} - q_n}{n^\theta} < \infty}.
\]
The conjecture we stated implies that $\theta^* = 0$, though this too is wide open. We remark
that assuming the $abc$-conjecture, it was shown by Granville~\cite{MR1654759} that $\theta^* = 0$.

In this paper we show that $\theta^*\le 1/5 - \eta$ for some $\eta > 0$. In fact we are also
able to show an asymptotic for the number of squarefree numbers in a short interval
of such a short length in the following result:
\begin{theorem}\label{thm:main}
  There exists $\eta > 0$ such that for $X^{\frac 15 - \eta}\ll H\le X$, we have
  \[
    \sum_{X\le n\le X + H} \mu^2(n) = \frac{6}{\pi^2}H(1 + O(X^{-\eta})).
  \]
\end{theorem}
This represents the first improvement on $\theta^*$ since\footnote{
  There is a preprint of Tsz Ho Chan~\cite{2021arXiv2110.09990} claiming to show that
  $\theta^*\le \frac{5}{26}$, though there is an error in this preprint according to its
  author (private communication).}
work of Filaseta and Trifonov~\cite{MR1171549} from 1992 in which it is shown that $\theta^*\le 1/5$.

We have left $\eta$ inexplicit, though in principle it should be possible to record a specific value.
This would require the exponents in our application of Green and Tao's quantitative Leibman
theorem \cite{GT} to be made explicit. We decided not to do so as it would significantly
lengthen the paper without producing qualitatively significant saving, while our main goal is
go past the natural threshold of $1/5$. We refer the reader to \S\ref{sec:15barrier_explain} for a discussion of
why $1/5$ is natural from the perspective of technique: it appears to be the best one can do with a single differencing identity.

The work of Filaseta--Trifonov followed work of Fogels~\cite{MR4843}, Roth~\cite{MR43119},
Richert~\cite{10.1112/jlms/s1-29.1.16}, Rankin~\cite{MR72170}, Schmidt~\cite{Sch}, Graham--Kolesnik~\cite{MR967330},
Trifonov~\cite{MR985889, MR1048051}, and Filaseta~\cite{MR1057318}. These works used both elementary methods
and exponential sums, with the $1/5$ bound of Filaseta--Trifonov~\cite{MR1171549} following from
purely elementary methods.

Our starting point is Filaseta and Trifonov's proof arranged in
such a way that the bound $\theta^*\le 1/5$ follows from a trivial bound for a certain
family of point counting questions. We deal with all ranges but one elementarily with
several new methods. For the critical range left, with further elementary manipulations,
we are reduced to counting lattice points satisfying certain constraints close to a curve.
It is in dealing with this range that we are led to require an understanding of the distribution
modulo $1$ of certain bracket polynomials, which we obtain from Green and Tao's work \cite{GT}.


The technique used to deal with this final critical range should lead to asymptotics for
the number of lattice points near curves of high dimension satisfying mild conditions.
We intend to make this precise in future work.

We end this section noting that there is a result of Bombieri and Bourgain
\cite{MR2505444} on ultraflat polynomials which uses
Filaseta and Trifonov's $1/5$ exponent at one point. Interestingly, this is one of the bottlenecks
to improving a certain exponent in their paper.
See the discussion at the end of~\cite[\S1]{MR2505444} for more on this.

\subsection{Notation and conventions}

We let $\norm{x} = \inf_{n\in\ZZ} |x - n|$. Also, we write $a\sim A$ to denote $A < a\le eA$.

We use Vinogradov notation $\ll, \gg, O(-), \asymp$ as usual. 
We make precise below the specific conventions regarding dependencies
between implied constants in hypotheses and statements. All implied constants are
absolute unless otherwise specified, and any parameters on which they might depend
will be specified with a subscript.

Implied constants in statements depend on implied constants in hypotheses.
For example, suppose we have a statement of the form
\[
  A_1\ll B_1,\dots,A_k\ll B_k\implies X_1\ll Y_1,\dots,X_\ell\ll Y_\ell.
\]
This should be read as meaning that for all $c_1,\dots, c_k > 0$, there exist
$c_1',\dots, c_\ell' > 0$ such that if $|A_1|\le c_1 B_1,\dots, |A_k|\le c_kB_k$,
then $|X_1|\le c_1' Y_1,\dots, |X_\ell|\le c_\ell' Y_\ell$.

We also introduce the variant $\lll$ of $\ll$, writing $A\lll B$ to denote
that $|A|\le cB$ for some sufficiently small $c > 0$. For example, the statement
$A\lll B\implies X\ll Y$ should be read as: for $c > 0$ sufficiently small, there exists
$c' > 0$ such that $|A|\le cB\implies |X|\le c'|Y|$.

In Table \ref{tab:cq5rbqwg57}, we have included a list of symbols used in the paper and their definitions
along with a little more information for ease of navigation of the paper.

\subsubsection{Archaic Greek characters}

At a couple points, we shall use the archaic Greek characters $\Sampi$ (Sampi) and \Qoppa (Qoppa).
Commands for these in \LaTeX~are as follows
(with documentclass \lstinline{amsart}; with \lstinline{article} they do not work):
\begin{verbatim}
\DeclareFontFamily{U}{cbgreek}{}
\DeclareFontShape{U}{cbgreek}{m}{n}{<10-12>grmn1000<12-17>grmn1200<17->grmn1728}{}
\DeclareRobustCommand{\Qoppa}{\text{\usefont{U}{cbgreek}{}{n}\symbol{21}}}
\DeclareRobustCommand{\Sampi}{\text{\usefont{U}{cbgreek}{}{n}\symbol{23}}}
\end{verbatim}
\newpage

\newpage
\section{Setup and outline of the proof}\label{sec:cq5rbolqen}

We take $\eta > 0$ fixed and sufficiently small throughout. We take $H\asymp X^{\frac{1}{5} - \eta}$ and write
$G = X/H^5 \asymp X^{5\eta}$ and also take $U = X^{\eta}$, though our statements hold so long as $U\ll G$.


\subsection{Reduction to squarefree sieve bounds in short intervals and
  proof of the main theorem}\label{subsec:reduction}

We succeed in obtaining the main theorem by improving bounds for the number of integers in $[X, X + H]$ that are divisible by squares of certain sizes.
Let $D_- = H/U, D_+ = 2X^{\frac 12}$. Then, 
\begin{multline}\label{eq:mob_exp}
  \sum_{X\le n\le X + H}\mu^2(n) 
  = \sum_{X\le n\le X + H}\sum_{d^2 | n} \mu(d) = \sum_{d\le D_+}\mu(d)\ssum{X\le n\le X + H\\ d^2 | n} 1
  \\ = \sum_{d\le D_-} \mu(d)\biggl(\frac{H}{d^2} + O(1)\biggr) + O\biggl(\sum_{D_-\le D \le D_+} \#\mc D_{[D, 2D]}\biggr)
  \\ = \frac{6}{\pi^2} H + O\biggl(D_- + \frac{H}{D_-} + \sum_{D_-\le D\le D_+}\#\mc D_{[D, 2D]}\biggr),
\end{multline}
where $D$ runs over integers of the form $2^k D_- $ with $k\ge 0 $ and we write
\begin{equation}\label{eq:cq5rbpe0xy}
  \mc D_S = \set{d\in S : \text{There exists } m\in \mathbb{Z}\text{ such that } md^2\in [X, X + H]}.
\end{equation}
Here, we have used the fact that for $d\gg\sqrt{H}$, at most $O(1)$-many $X\le n\le X + H$ can be divisible by $d^2$.
The $D_-, H/D_-$ terms in the error are admissible, so it remains to bound $\#\mc D_{[D, 2D]}$.

This is the content of the following two estimates, which may be of independent interest.
Before we state them, we consider one choice of $D_- \ll D\ll D_+$ and write $\Delta = D/H$.
Since the choice of $D$ remains stable, we write $\mc D = \mc D_{[D, 2D]}$.

\begin{proposition}\label{prop:all_new_bounds}

  If $\#\mc D\ge H/U$, we have the bounds
  \begin{equation}
    \label{eq:lower_regime_est} 
    \#\mc D\ll \frac{H}{U} + H
    \biggl( \frac{1}{\Delta^{\frac{1}{2}}}G^{9}U^{53}
    + \frac{\Delta^2}{H}G^{\frac{45}{2}}U^{\frac{243}{2}}
    + \frac{1}{\Delta}G^{19}U^{96} \biggr)\log H,
  \end{equation}
  \begin{equation}
    \label{eq:upper_regime_easy}
    \#{\cal D}\ll \frac{H}{U} + H\biggl(\frac{H}{\Delta^2}GU^2 + \frac{H^{\frac 12}}{\Delta}G^{\frac{3}{2}}
    U^{\frac{9}{2}} + \frac{1}{H^{\frac{1}{2}}}G^{\frac{1}{2}}U^{\frac{5}{2}}\biggr).
  \end{equation}
  Furthermore, there exists $\sigma > 0$ such that if $G, U, H/\Delta^2, \Delta^2/H\ll X^{10^{-5}}$, we have 
  \begin{equation}\label{eq:upper_regime_est}
    \#\mc D\ll \frac{H}{U} + H\biggl(1 + \frac{H}{\Delta^2}\biggr)X^{-\sigma}(GU)^{O(1)}.
  \end{equation}
\end{proposition}
We shall also require an estimate for $\Delta$ near $H$.
This will be the following consequence of a general result of Filaseta--Trifonov~\cite{MR1397690}
from an iteration of an elementary differencing procedure.
\begin{proposition}\label{prop:very_low_est}
  If $D\lll (X/H)^{\frac 13}$, then we have that
  \[
    \#\mc D\ll X^{\frac{1}{10}}D^{\frac 25} + X^{\frac{3}{20}}D^{\frac 18}.
  \]
\end{proposition}

The starting point of all estimates in Proposition \ref{prop:all_new_bounds}
is a better understanding of the structure of $\mc D$. Implicitly, as we shall discuss, this is related to the proof in Filaseta and Trifonov's proof in \cite{MR1171549} that $\theta^*\le 1/5$.
The trivial bound from this starting point recovers essentially their bound.

It has since been brought to our attention that there is an improvement of \cite{MR1397690} due to Huxley and Sargos \cite{MR1318755}. 
This would slightly simplify the proof of Proposition~\ref{prop:very_low_est},
though both provide the bound $\#{\cal D}\ll X^{\frac 15 - \frac{2}{5}\delta}$ for $D\ll X^{\frac 14 - \delta}$.

Before we set up and outline the proofs of these results, we shall prove the main theorem conditional on Propositions \ref{prop:all_new_bounds} and \ref{prop:very_low_est}.

\begin{proof}[Proof of Theorem \ref{thm:main}]
  We shall show Theorem~\ref{thm:main} with an error term of $O(X^{-\eta + o(1)})$, which
  will clearly be sufficient.
  
  By Proposition~\ref{prop:very_low_est}, for $1/U\ll \Delta\ll X^{\frac{1}{100}}$, we have that
  \begin{equation}
    \label{eq:D_bound_final_proof}
    \#\mc D_{[H\Delta, 2H\Delta]}\ll \frac HU.
  \end{equation}
  By \eqref{eq:lower_regime_est} and \eqref{eq:upper_regime_easy}, we have that \eqref{eq:D_bound_final_proof} holds for $\Delta\in [X^{\frac{1}{100}}, H^{\frac{1}{2}}U^{-1000}]\cup [H^{\frac{1}{2}}U^{13}, X^{\frac{1}{2}}/H]$. 
  For $H^{\frac 12}U^{-1000}\ll \Delta\ll H^{\frac 12}U^{13}$, we have 
  \[
    \#{\cal D}\ll HU^{O(1)}X^{-\sigma}.
  \]
  For $\eta$ sufficiently small, this is $\ll H/U$. It follows from \eqref{eq:mob_exp} that
  \[
    \sum_{X\le n\le X + H} \mu^2(n) = \frac{6}{\pi^2}H + O\biggl(\frac HU\log X\biggr)
    = \frac{6}{\pi^2}H(1 + O(X^{-\eta + o(1)})),
  \]
  for $H\asymp X^{\frac 15 - \eta}$, from which the desired result follows (perhaps with a smaller $\eta $).
  The result for larger $H$ follows by averaging over shorter
  intervals. 
\end{proof}

\subsection{Proof of Proposition \ref{prop:very_low_est}}
\label{sec:verylow}

Proposition~\ref{prop:very_low_est} is a direct consequence of the following result of Filaseta and
Trifonov \cite{MR1397690}:
\begin{theorem}[{\cite[Theorem 6]{MR1397690}}]
  \label{thm:FT_very_low_bound}
  Fix some integer $r\ge 3$.
  Suppose that $N > 1, T, \delta > 0$ and that $f$ has $r$ derivatives with
  $f^{(j)}(x)\asymp TN^{-j}$ for $j\in\set{r-2, r-1, r}$.
  Furthermore, suppose that we have 
  \[
    \delta \le  c_r\min\set*{TN^{-r + 2}, T^{\frac{r - 4}{r - 2}}N^{-r + 3} + TN^{-r + 1}}
  \]
  for $c_r > 0$ sufficiently small (dependent also on the choice of implied constants
  in the bounds $f^{(j)}(x)\asymp TN^{-j}$). 
  Then, we have that
  \[
    \#\set{n\sim N : \norm{f(n)}\le\delta} \ll T^{\frac{2}{r(r + 1)}}N^{\frac{r - 1}{r + 1}}
    + N\delta^{\frac{2}{(r - 1)(r - 2)}}
    + N(\delta TN^{1 - r})^{\frac{1}{r^2 -3r + 4}}.
  \]
\end{theorem}

Taking $T = X/D^2, N = D, \delta = H/D^2, f(n) = X/n^2, r = 4$, it can be checked that the
hypotheses of Theorem \ref{thm:FT_very_low_bound} hold as long as
\[
  \frac{H}{D^2} \le c \frac{X}{D^5}
\]
for some sufficiently small $c > 0$. Rearranging, this follows from the hypothesis that
$D\lll (X/H)^{\frac 13}$. We therefore obtain the bound
\[
  \#\set{d\sim D : \norm{\frac{X}{d^2}}\le \frac{H}{D^2}} \ll X^{\frac{1}{10}}D^{\frac 25} + D^{\frac 13}H^{\frac 13} + D\pfrc{HX}{D^7}^{\frac 18}.
\]
Since $H\ll X^{\frac 15}$, it can be checked that the second term may be absorbed into the first.
The desired result follows.

The remainder of this paper will be devoted to showing Proposition \ref{prop:all_new_bounds},
and the next few subsections below shall be setup common to all parts of Proposition
\ref{prop:all_new_bounds}. Then, in \S\ref{sec:lower} we show \eqref{eq:lower_regime_est} and in \S\ref{sec:upper} we show \eqref{eq:upper_regime_easy} and \eqref{eq:upper_regime_est}.

\subsection{Initial maneuvers for Proposition \ref{prop:all_new_bounds}}
Write
\begin{equation}\label{eq:cq5rbpwe1z}
  \mc D_a = \set{d : d, d + a\in \mathcal{D}\text{ such that } (d, d + a)\cap \mathcal{D} = \emptyset}
\end{equation}
By the bound $\#\mathcal{D} \ge H/U$ we have as a hypothesis in Proposition~\ref{prop:all_new_bounds}, we
have that
\[
  \#\mc D\ll \frac{H}{U} + \sum_{1\le A\le \Delta U}\sum_{a\sim A} \#\mc D_a + \frac{D}{\Delta U} \ll \frac{H}{U} + \sum_{1\le A\le \Delta U}\sum_{a\sim A} \#\mc D_a,
\]
where $A$ runs over powers of $e$.

We may immediately discard from consideration some ranges of $A$. By an observation of
Nair~\cite{MR531167} following work of Roth \cite{MR43119}, we have that among any three consecutive
elements of $\mc D$, two are spaced by $\gg \Delta^{\frac 43}(H^4/X)^{\frac 13}$. In particular, we obtain the
following, which is \cite[(8)]{MR1171549}:
\begin{equation}\label{eq:pigeonholed}
  \#\mc D\ll 1 + \frac{H}{U} + \sum_{1 + \Delta^{\frac 43}(H^4/X)^{\frac 13}\ll A\ll\Delta U} \sum_{a\sim A} \#\mc D_a.
\end{equation}
From now on, we fix some scale $A$ in the above range. The main difficulties will
be encountered for $A$ at the upper end of the range.

\subsection{Differencing identities}\label{sec:cq5rbqt63z}
\label{sec:diff_identities}
Key to elementary progress on the problem of bounding $\theta^*$ have been ``differencing identities''.
The basic idea starts with the observation that
the condition $d\in\mc D$ implies (and is essentially equivalent to) the condition
\[
  \norm{\frac{X}{d^2}}\ll \frac{H}{D^2} = \frac{1}{H\Delta^2}.
\]
Assuming $\#{\cal D}$ is large, the idea then is to take well-chosen $\ZZ$-linear combinations
of this identity to obtain small quantities near integers. Then, by exploiting the fact
that an integer of size less than $1$ must be $0$, further information can be obtained.
The hope is that this ends with a contradiction if ${\cal D}$ is too close to $H$ in size.

The most basic differencing identity, which, used appropriately, gives a bound of
$\theta^*\le 1/3$ is $F_a(d)$, defined as
\begin{equation}\label{eq:cq5rbqjd9g}
  F_a(d) := \frac{X}{d^2} - \frac{X}{(d + a)^2}.
\end{equation}
Then, we have that $\norm{F_a(d)}\ll \frac{1}{H\Delta^2}$. The virtue of this is that $F_a(d)$ is
smaller than $X/d^2$ for $a$ small. More precisely, from the Taylor expansion 
\begin{equation}\label{eq:taylor_std}
  \frac{1}{(1 + x)^2} = 1 - 2x + 3x^2 + \dots,
\end{equation}
we have for $|a| < d $ that
\begin{equation}\label{eq:std_exp}
  F_a(d) = \frac{Xa}{d^3}\biggl(2 - 3\frac ad + \dots\biggr).
\end{equation}
Key to this paper shall be the following identity from work of Roth~\cite{MR43119}
(which was used to give a bound of $\theta^*\le 1/4$). Let
\begin{equation}\label{eq:cq5rbql9yk}
  R_a(d) := -(2d - a)\frac{X}{d^2} + (2d + 3a)\frac{X}{(d + a)^2}.
\end{equation}
Note that if $d, d + a\in\mc D$, we have that $\norm{R_a(d)}\ll 1/\Delta$.
The expansion of $R_a(d)$ we require is as follows:
\begin{equation}\label{eq:roth_exp}
  R_a(d) = \frac{Xa^3}{d^2(d + a)^2} = \frac{Xa^3}{d^4}\biggl(1 - 2\frac ad + \dots\biggr).
\end{equation}
The fact that $U\le H^{\frac 12}$ implies that the above series converges for $a\sim A$ as we have
$A\ll U\Delta\ll D/H^{\frac 12}$. 
We record here some implications of this which shall be important in what follows.

We shall regard $R_a$ as a map from $(D/2, 5D/2)$ to $\mathbb{R}$, in which case it is injective, and
let $\tilde d_a := R_a^{-1}$ be defined on $\rho\asymp R$ with 
\begin{equation}\label{eq:11}
  R := \frac{XA^3}{\Delta^4H^4}.
\end{equation}
From the series \eqref{eq:roth_exp}, for $j\ge 0$ and $\rho \asymp R$ we have 
\begin{equation}
  \label{eq:tilded_ests}
  \tilde d_a^{(j)}(\rho)\asymp_j \frac{H\Delta}{R^j}.
\end{equation}
It follows that for $r, r'\asymp  R$, we have
\begin{equation}
  \label{eq:tilded_appro_inv}
  \bigg|\frac{\tilde d_a(r)}{\tilde d_a(r')} - 1\bigg|\ll \bigg|\frac{r}{r'} - 1\bigg|.
\end{equation}

\subsection{An approximate parametrization of \texorpdfstring{$\mc D$}{}}

We shall obtain a better understanding of ${\cal D}_a$ through an
``approximate parametrization'' of $\mc D$ coming from Roth's differencing.
The result is as follows:
\begin{proposition}\label{prop:approx_param}

  There exist $\mc R_a\subset\NN$ with $r\asymp R$ for all $r\in{\cal R}_a$ and a map
  \begin{equation}\label{eq:4}
    d_a^* : \mc R_a\to\mc D_a
  \end{equation}
  satisfying the following two bounds:
  \begin{equation}\label{eq:D_aR_a_size_bound}
    \frac{\#\mc D_a}{\#\mc R_a}\ll 1 + \biggl(\frac{\Delta}{A}\biggr)^{\frac 83}\frac{1}{G^{\frac 23}},
  \end{equation}
  \begin{equation} \label{eq:d_approx}
    d_a^*(r) = \tilde d_a(r) + O\biggl(\frac{\Delta}{G}\frac{\Delta^3}{A^3}\biggr).
  \end{equation}
\end{proposition}

We will require the following spacing result for $\mc D_a$ to show Proposition~\ref{prop:approx_param}.
This follows from elementary considerations involving Roth's differencing.
\begin{lemma}\label{lem:Da_space}
  We have that consecutive elements of $\mc D_a$ are at least $a^{-\frac 13}\Delta^{\frac 53}(H^5/X)^{\frac 13}$-spaced.
\end{lemma}

\begin{proof}
  See, for example, \cite[\S2]{MR1171549} for a proof of this and a discussion of previous incarnations of this estimate.
\end{proof}

\begin{proof}[Proof of Proposition \ref{prop:approx_param}]
  Note that for all $d\in\mc D_a$, we have that $\| R_a(d) \|\ll 1/\Delta $, and so there exists a
  map $r_a^* : \mathcal{D}_a\to \mathbb{N}  $ such that 
  \begin{equation}\label{eq:3}
    R_a(d) = r^*_a(d) + O\pfrc{1}{\Delta}
  \end{equation}
  for all $d\in \mathcal{D}_a $. We will also take $\mathcal{R}_a = r_a^*(\mathcal{D}_a) $ and take any $d_a^* : \mathcal{R}_a\to \mathcal{D}_a$ such that
  $r_a^*\circ d_a^* = \mathrm{id}_{\mathcal{R}_a} $. It is easy to see that $r\in \mathcal{R}_a $ satisfy $r\asymp R $ in light of
  \eqref{eq:3} and the fact that $R_a(d)\asymp XA^3/D^4\asymp R $. 
  
  We claim that for any $r\in \mathcal{R}_a $, the preimage $(r_a^*)^{-1}(\set{r})$ has size
  \[
    \ll 1 + \biggl(\frac{\Delta}{A}\biggr)^{\frac 83}\biggl(\frac{H^5}{X}\biggr)^{\frac 23}.
  \]

  To see this, suppose for some $k\ge 2$, there exist distinct
  $d_1< \dots< d_k$ such that $r_a^*(d_1) =\dots = r_a^*(d_k)$.
  Then, we have 
  \begin{equation}\label{eq:R_sep}
    |R_a(d_1) - R_a(d_k)|\ll \frac{1}{\Delta}.
  \end{equation}
  On the other hand, by \eqref{eq:roth_exp}, we have
  \[
    |R_a(d_1) - R_a(d_k)|\asymp \frac{XA^3}{(H\Delta)^4}\cdot \frac{|d_1 - d_k|}{H\Delta}.
  \]
  By Lemma \ref{lem:Da_space}, the above is
  \[
    \gg \frac{XA^3}{\Delta^4 H^4} k\frac{a^{-\frac 13}\Delta^{\frac 53} (H^5/X)^{\frac 13}}{H\Delta}.
  \]
  If
  \[
    k\ge C\biggl(\frac{\Delta}{A}\biggr)^{\frac 83}\biggl(\frac{H^5}{X}\biggr)^{\frac 23}
  \]
  for some $C > 0$, then we have
  \[
    |R_a(d_1) - R_a(d_k)|\gg \frac{C}{\Delta}.
  \]
  This contradicts \eqref{eq:R_sep}; hence
  \[
    k\ll 1 + \biggl(\frac{\Delta}{A}\biggr)^{\frac 83}\biggl(\frac{H^5}{X}\biggr)^{\frac 23}.
  \]
  The claimed bound on the preimage size follows.

  The bound on the size of the preimage implies the bound \eqref{eq:D_aR_a_size_bound} of
  Proposition~\ref{prop:approx_param}.

  It remains to verify the estimate \eqref{eq:d_approx} on $d_a^*$. To see this, note that
  for $r\in\mc R_a$, we have
  \[
    r = R_a(d_a^*(r)) + O\pfrc{1}{\Delta}
    = R_a(d_a^*(r))\biggl(1 + O \biggl(\frac{1}{R\Delta}\biggr)\biggr).
  \]
  Rearranging to solve for $d_a^*(r)$ using \eqref{eq:tilded_appro_inv} yields that
  \[
    d_a^*(r) = \tilde d_a(r)\biggl(1 + O \biggl(\frac{1}{R\Delta}\biggr)\biggr),
  \]
  as desired, since it can be checked that $\frac{D}{R\Delta}\asymp \frac{\Delta}{G\Omega^3}$.
\end{proof}

\subsection{Outline of the proof of Proposition \ref{prop:all_new_bounds}} \label{sec:outline_setup}
At this point, applying the trivial bound corresponds roughly to the argument of
Filaseta--Trifonov \cite{MR1171549} and recovers the bound $\theta^*\le 1/5$. We will succeed in going
past $1/5$ by obtaining restrictions on the size of $\mc R_a$, with different treatments for
$\Delta$ in different ranges relative to $H^{\frac{1}{2}}$.

For small $\Delta$, we use new differencing identities to obtain finer information about
$\mc D_a$ than is implied by Proposition \ref{prop:approx_param}. More precisely, Proposition \ref{prop:approx_param} implies that
for $\Delta$ small, if $\mc R_a$ has size near $R$, then $\mc D_a$ contains many approximate linear patterns
by the pigeonhole principle. We use a differencing identity to obtain nontrivial
information about the defect of these approximate patterns from exact linear patterns.
Another differencing identity then uses this information to yield restrictions on $r$ which
are exploited by the elementary point counting estimates we will show in \S\ref{sec:minor_lemmas}.

For large $\Delta$, we succeed by utilizing Proposition \ref{prop:approx_param} to obtain an estimate for
$F_a(d_a^*(r))$. When $H = X^{\frac 15}$, at the critical point, we obtain an estimate $\tilde F$ for
$F_a(d_a^*(r))$ to an error of $O(1)$. We know that $F_a(d_a^*(r))$ is within $\frac{1}{H\Delta^2}$ of an integer, so
up to $O(1)$-many values, we know $F_a(d^*_a(r))$ in terms of $\floor*{\tilde F}$. Then, using \eqref{eq:std_exp}, we
obtain an even more accurate estimate for $d^*_a(r)$, one which must be close to an integer.
If we split the sum over $r$ into short intervals and Taylor expand, we are reduced to
checking that a certain bracket polynomial is close to an integer a small proportion of
the time. It is at this point we utilize the estimates in \S\ref{sec:bracket_ests} to show that this cannot
happen a $\gg X^{-o(1)}$-fraction of the time. It is here that we require Green and Tao's \cite{GT}
quantitative Leibman theorem.  


\subsection{Final reductions}
\label{sec:final_reductions}

Combining the trivial bound $\#\mathcal{R}_a\ll R $ with Proposition~\ref{prop:approx_param} implies that
\begin{equation}
  \#\mathcal{D}_a\ll \#\mathcal{R}_a \biggl( 1 + \biggl(\frac{\Delta}{A}\biggr)^{\frac{8}{3}}\frac{1}{G^{\frac{2}{3}}} \biggr)
  \ll R \biggl( 1 + \biggl(\frac{\Delta}{A}\biggr)^{\frac{8}{3}}\frac{1}{G^{\frac{2}{3}}} \biggr).
\end{equation}
After absorbing the contribution of the small-$A$ range into the $H/U$ term in \eqref{eq:pigeonholed}, we have
\[
  \#\mathcal D\ll \frac HU+
  \sum_{\Delta G^{-1/4}U^{-3/4} \ll A \ll \Delta U}
  \sum_{a\sim A}\#\mathcal D_a.
\]

We shall therefore suppose from this point on that $A\gg \Delta G^{-\frac{1}{4}} U^{-\frac{3}{4}} $.
For slight convenience in gathering certain bounds, we write
$A/\Delta = \Omega$ so that
\begin{equation}\label{eq:cq6ufgf398}
  G^{-\frac{1}{4}}U^{-\frac{3}{4}}\ll \Omega \ll U.
\end{equation}
\subsection{On the exponent $1/5$ and differencing identities}
\label{sec:15barrier_explain}

In this section, we will discuss heuristically why $1/5$ appears to be the limit of
what one can obtain from differencing identities (as introduced in \S\ref{sec:diff_identities}).
For the purposes of this discussion, we will consider how one might go about ruling out
the situation that $\#{\cal D}\asymp H$ (that the upper bound holds follows from the divisor bound)
for $1\ll\Delta\ll X^{\frac 12}/H$. It is to this that counting squarefree numbers in $[X, X + H]$
essentially boils down.

We will also restrict ourselves to the situation that ${\cal D}$ is $\Delta$-spaced, and suppose too
that ${\cal D}_a$ is $\Delta^2$-spaced for $a\sim\Delta$ (generically, one should expect this to be the case).

We shall outline the main input behind the proofs of the upper bounds $1/3$, $1/4$, $1/5$ for
$\theta^*$ due to Estermann (described in \cite{MR43119}), Roth~\cite{MR43119}, and Filaseta-Trifonov~\cite{MR1171549}
respectively.

Recall that the basic idea behind differencing identities is to lower the size of
appropriate $\ZZ$-linear combinations of various $X/d^2$ for $d\in{\cal D}$. One way to succeed is to
observe that if we can control the size enough that we know it is $\lll 1$, but still larger
than its distance to the nearest integer, we obtain a contradiction.

The most basic example of a differencing identity is to simply take the difference of two nearby
elements in ${\cal D}$. More precisely, recall
\[
  F_a(d) = \frac X{d^2} - \frac{X}{(d + a)^2}.
\]
Consideration of this yields $\theta^*\le 1/3$, as we outline now.
If we take $d, d + a\in\mc D$, then $\norm*{F_a(d)}\ll 1/(H\Delta^2)$. If $a\sim\Delta$, then by Taylor expanding
\[
  F_a(d)\asymp \frac{X}{H^3\Delta^2}.
\]
One can check this is $\ggg 1/(H\Delta^2)$ and $\lll 1$ when $X^{1/3}\lll H\lll X^{1/2}$, which is a contradiction,
yielding that $\theta^*\le 1/3$. This is essentially the argument of Estermann, as described by Roth~\cite{MR43119}

Similarly, the differencing identity $R_a(d)$ (refer to \S2.4 for its definition) has size $X/(H^4\Delta)$, and
thus yields nontrivial information in this fashion when $X\ll H^4\Delta$ (and thus\footnote{The argument at this point is somewhat more tricky than as above, but still requires only the information coming from the differencing identity} yields that $\theta^*\le 1/4$).
This is the main input to the argument of Roth~\cite{MR43119}.

The improvements over Roth's bound (before Filaseta and Trifonov in the papers \cite {MR1057318} and \cite{MR1171549} returned with elementary methods)
used exponential sums to deal with small
$\Delta$ not finished off by Roth's elementary method. Such methods are in practice limited to small
$\Delta$, since as $\Delta$ grows, they demand greater and greater savings from exponential sums. Large $\Delta$
are a somewhat different story, but wide intermediate ranges were out of reach (and still remain
so with exponential sum methods, barring massive improvements in techniques for bounding exponential
sums).

We will now sketch, in language closer to the original proof than our setup, the idea of Filaseta and Trifonov.
Another differencing identity one can consider comes from successive elements $d, d + b\in\mc D_a$.
With the Pigeonhole principle, we can ensure $b\ll \Delta^2$ for many $b$ for which the above holds
(showing lower bounds on $b$, and in particular that it is $\sim \Delta^2$ is actually quite nontrivial,
though we will suppose it is so in this discussion for simplicity; this is essentially the content of
our Lemma~\ref{lem:Da_space}, which was shown in \cite{MR1171549}).

Now, consider the differencing
\begin{equation}\label{eq:5}
  R(d, a, b) = R_a(d) - R_a(d + b).
\end{equation}
For $d, d + b\in{\cal D}_a$, since $d, d + a, d + b, d + a + b\in\mc D$, we have $\norm{R(d, a, b)}\ll 1/\Delta$.
Also, it can be checked that 
\[
  R(d, a, b)\asymp \frac{Xa^3b}{d^5}\asymp \frac{X}{H^5}.
\]
This yields a contradiction when $H$ is slightly larger than $X^{1/5}$.

One might hope there exist differencing identities which alone are capable of going
past $1/5$, but this seems not possible in the range $H^{\frac 13}\ll\Delta\ll H^{\frac 23}$, especially in the
critical case $H = X^{\frac 15}$. Given any configuration in ${\cal D}$ one might generate by
the Pigeonhole principle (such as $d, d + a, d + b, d + a + b$ or $d, d + a_1, d + a_2$, for example),
one can find the best possible differencing identity by solving a system trying to cancel
out terms in the Taylor expansion of $X/d_1^2$ for $d_1$ near $d$.

One can try such maneuvers as fibering over $b\sim \Delta^2$ (constructing sets such as
${\cal D}_{a, b} = \set{d : d, d + b\text{ consecutive elements of }{\cal D}_a}$ and bounding $\#{\cal D}_{a, b}$), which provides new opportunities for differencing
identities. This is immediately doomed to fail for $\Delta\gg H^{\frac 13}$ as we cannot hope for a bound on ${\cal D}$ better than the number of $(a, b)$,
which is $\asymp\Delta^3$ (and succeeding for $\Delta$ as large as $H^{\frac 13}$ seems a challenge on its own).  

One can also seek to build differencing identities out of configurations
$d, d + a_1, \dots, d + a_k\in{\cal D}$ for $a_i\sim\Delta$. In this case, it also seems a differencing
identity with coefficients that are polynomials in $d, a_1,\dots,a_k$ (with monomials that are
not unacceptably large) is not possible unless $\Delta$ is very small.
\footnote{Interestingly, the range of $\Delta$ in which one can succeed here for $H = X^{o(1)}$ is stronger than that which
  Vinogradov's exponential sum bounds yield.}
In the case of $H = X^{1/5}$, it seems this direction is only capable of producing improvements when $\Delta\lll H^{1/4}$ (and, as before,
this can be checked by solving a system in the coefficients of the monomials in the differencing identity).

Combining these two ideas, such as by seeking to build differencing identities of the out of configurations like
$d, d + a_1, d + a_2, d + b, d + b + a_1, d + b + a_2$ with $a_i\sim\Delta, b\sim\Delta^3$ also similarly seems incapable of saying anything new.

Our methods are able to go beyond what an individual differencing identity can yield by more carefully studying the structure of ${\cal D}$
starting with our ``approximate parametrization'' Proposition~\ref{prop:approx_param}, as we outlined in more detail
at the end of \S\ref{sec:outline_setup}. 

It turns out that our two elementary estimates \eqref{eq:lower_regime_est} and \eqref{eq:upper_regime_easy} fail at
precisely the range $H = X^{1/5}, \Delta = H^{1/2}$, and we are required to boost \eqref{eq:upper_regime_est} by using the quantitative
Leibman theorem of Green and Tao \cite{GT}. It seems there is something critical at these ranges, though we
do not have any deeper explanation or justification for this claim of criticality.

\section{Elementary estimates on points near curves}
\label{sec:minor_lemmas}
In this section, we record some minor technical lemmas that will be used several times
in the proof of Proposition \ref{prop:all_new_bounds}.

\begin{lemma}\label{lem:crude_frac_bound}
  For some interval $I$, suppose that $\phi \in C^1(I) $ satisfies
  \[
    \bigg|\sup_{x\in I} \phi(x) - \inf_{x\in I}\phi(x)\bigg|\ll V, |\phi'(x)|\asymp F
  \]
  for $x\in I$.
  Then, uniformly in $0 < \delta < 1$, we have
  \[
    \sum_{n\in I}\charf{\norm{\phi(n)}\le\delta}\ll (V + 1)\biggl(1 + \frac{\delta}{F}\biggr). 
  \]
\end{lemma}
\begin{proof}

  Since $V/F\gg |I| $, the result is trivial if $\delta\gg 1 $, so suppose it is $< 1/2 $ in what remains.
  
  Due to the derivative bound on $\phi$, for any interval $J$, we have
  \begin{equation}\label{eq:18}
    |\phi^{-1} (J)|\ll \frac{1}{F} |J|.
  \end{equation}
  We also have
  \[
    \sum_{n\in I}\charf{\norm{\phi(n)}\le\delta}\le\sum_{m\in\ZZ}\sum_{n\in I}\charf{\phi(n)\in [m - \delta, m + \delta]} = \sum_{m\in\ZZ} \sum_{n\in \phi^{-1}([m - \delta, m + \delta])}1 .
  \]
  Because of the upper bound on the variation of $\phi$, the above sum is over $m$ supported
  on $m$ in an interval of length $\ll V + 1$. Along with \eqref{eq:18}, it follows that the
  above is $\ll (V + 1)(1 + \delta/F)$.
\end{proof}

We shall also require the following variant which allows for the derivative of $\phi$ to be small at a point.
This captures the situation in which $\phi$ is approximately the sum of two monomials.
\begin{lemma}\label{lem:crude_stat_est}
  Suppose that $I\subset [c_1N, c_2N]$ for some fixed $c_1 < c_2$ and that $\phi\in C^2(I) $ satisfies
  \begin{equation*}
    |\phi'(x)| \ll \frac{T}{N}, |\phi'(x)| + N|\phi''(x)|\asymp \frac{T}{N}
  \end{equation*}
  for $x\in I$. Suppose also that $\phi', \phi''$ have at most $O(1)$-many zeros, and
  that $0 < \delta < 1$. 
  Then, we have
  \[
    \sum_{n\in I}\charf{\norm{\phi(n)}\le\delta}\ll N\biggl(\delta + \frac{\delta^{\frac 12}}{T^{\frac 12}}\biggr) + T + 1.
  \]
\end{lemma}
It can be checked that the $\delta^{\frac 12}/T^{\frac 12}$ term is necessary: consider, for example
$\phi(n) = T\pfrc{n - N}{N}^2$, for any $\delta > 0$. 
\begin{proof}

  We may suppose that $\delta < 1 $ in what follows.
  
  Without loss of generality, suppose that $\phi''$ does not change sign
  (as we can split $I$ into $O(1)$-many subintervals to reduce to this case).
  Let $c\ll 1$ be such that $|\phi'(x)|\le c\frac TN$ implies that $|\phi''(x)|\gg T/N^2$.
  Now, let $K = \ceil{\log(T/\delta)/2}$ and take 
  \begin{align*}
    I_> &= \set{x\in I :|\phi'(x)|\ge c\frac TN},\\
    I_k &= \set{x\in I : |\phi'(x)|\sim ce^{-k}\frac TN}, \\
    I_< &= \set{x\in I : |\phi'(x)| < ce^{-K}\frac TN}, 
  \end{align*}
  so that
  \[
    I = I_>\cup I_<\cup\biggl(\bigcup_{1\le k\le K} I_k\biggr).
  \]

  Note that the monotonicity of $\phi'$ implies that $I_>, I_k, I_<$ are all the union of up to two
  intervals. By the lower bound on $\phi''$ in $I_k$ implied for $k\ge 1$ by our earlier discussion,
  the length of $I_k$ is $\ll \frac{e^{-k} T/N}{T/N^2} = e^{-k} N$.
  By the bound on $\phi'$, $\phi$ varies by $\ll e^{-2k}T$ in $I_k$.
  Similarly, it varies by $\ll e^{-2K}T$ in $I_<$, and the length of  $I_<$ is $\ll e^{-K}N$.
  
  The latter implies by a trivial bound that 
  \[
    \sum_{n\in I_<}\charf{\norm{\phi(n)}\le\delta}\le 1 + e^{-K}N\ll 1 + N\sqrt{\frac{\delta}{T}}.
  \]
  For $k\le K$, we have by Lemma~\ref{lem:crude_frac_bound} that 
  \begin{align*} 
    \sum_{n\in I_k}\charf{\norm{\phi(n)}\le\delta} &\ll (1 + e^{-2k}T)\biggl(1 + e^k\frac{\delta}{T/N}\biggr) \ll 1 + e^{-2k}T + e^{-k}\delta N + e^k\frac{\delta}{T}N,\\
    \sum_{n\in I_>}\charf{\norm{\phi(n)}\le\delta} &\ll (1 + T)\biggl(1 + \frac{\delta}{T/N}\biggr)\ll 1 + T + \delta N + \frac{\delta}{T}N.
  \end{align*}
  Summing over $1\le k\le K$ and collecting the bounds for the sums over $I_<, I_>$ yields
  \[
    \sum_{n\in I}\charf{\norm{\phi(n)}\le\delta}\ll \delta N + e^K\frac{\delta}{T}N + T + 1\ll N\biggl(\delta + \sqrt{\frac{\delta}{T}}\biggr) + T + 1,
  \]
  as desired.
\end{proof}

At one point we shall require the following consequence of applying 
trivial bound after Fourier expanding the indicator function
and applying Poisson summation.

It is possible to show this elementarily similarly to the other results of this
section, but we have chosen not to do so for the sake of concision.
\begin{proposition}\label{prop:poisson_trivial}
  Suppose that $F: [1/3, 3]\to\RR$ is such that $|F''|\asymp 1$.
  Then, for any $T, N, \delta > 0$, we have 
  \[
    \#\set{n\sim N : \norm{T F\pfrc nN}\le \delta}\ll N\delta + \sqrt{\frac{T}{\delta}}
    + \frac{N}{\sqrt{T/\delta}}.
  \]
\end{proposition}
\begin{proof}
  This follows from \cite[Corollary 8.13]{IK} (the trivial bound after Poisson summation).
  To be precise, note that by Poisson summation, we have
  \[
    \charf{\norm{TF\pfrc nN}\le\delta}\le
    \sum_{m\in\ZZ}w\biggl(\delta^{-1}\biggl(TF\pfrc nN - m\biggr)\biggr)
    = \delta\hat w(0)
    + \delta\sum_{|\ell| > 0}\hat w(\delta\ell) e\biggl(\ell T F\pfrc nN\biggr) ,
  \]
  for $w$ any fixed smooth function with $\charf{[-1, 1]}\le w\le\charf{[-2, 2]}$, and then summing over $n$,
  moving the sum over $n$ inside, and applying \cite[Corollary 8.13]{IK} to the inner sum.
  
\end{proof}
\section{Proof of bounds nontrivial for $\Delta \ll H^{\frac{1}{2}}(GU)^{-O(1)}$ }\label{sec:cq5rctx73a}
\label{sec:lower}

In this section, we show the bound \eqref{eq:lower_regime_est}.
It will follow from the following bound on $\#\mc R_a$ with $\mathcal{R}_a$ satisfying the conditions of
Proposition~\ref{prop:approx_param}. To be precise, recall that
\begin{equation}
\mathcal{R}_a = \set{r_a^*(d) : d\in \mathcal{D}_a}
\end{equation}
where $\mathbb{Z}\ni r_a^*(d) \asymp R$ is defined by \eqref{eq:3}. 

\begin{proposition}\label{prop:Ra_bound_lower}
  Suppose $\Delta/(G^{\frac 14}U^{\frac 34})\ll a\ll \Delta U$, and that $\#\mc R_a\ge R/W$ for some $W\asymp GU^5$.
  Suppose also that $H/\Delta^2\ggg GU^{10}$ and that $\Delta\gg G^2U^5$.
  Then, we have that
  \[
    \#\mathcal{R}_a\ll \frac{H}{\Delta} \biggl( \frac{1}{\Delta^{\frac{1}{2}}} \frac{G^9U^{51}}{\Omega}
    + \frac{\Delta^2}{H} \frac{G^{16}U^{100}}{\Omega^{27}} + \frac{1}{\Delta} \frac{G^{16}U^{85}}{\Omega^{13}}\biggr),
  \]
  where we take $\Omega\asymp a/\Delta$.
\end{proposition}

\begin{proof}[Proof of \eqref{eq:lower_regime_est} assuming Proposition \ref{prop:Ra_bound_lower}]

  If either of the bounds $H/\Delta^2\gg GU^{10}, \Delta\gg G^2U^5$ are not satisfied,
  then \eqref{eq:lower_regime_est} follows trivially, so we may suppose we are in that regime.
  By our discussion in \S\ref{sec:final_reductions}, we have that
  \begin{equation}\label{eq:Dsum_a}
    \#\mc D \ll \frac{H}{U} + \sum_{\Delta G^{-\frac{1}{4}}U^{-\frac{3}{4}} \ll a\ll \Delta U}\#{\cal D}_a
  \end{equation}
  By Proposition \ref{prop:approx_param} and \eqref{eq:cq6ufgf398}, we have, using $a\gg \Delta/(G^{\frac{1}{4}}U^{\frac{3}{4}}) $, that
  \[
    \#\mc D_a \ll \#\mathcal{R}_a
    \biggl(1 + \biggl(\frac{\Delta}{a}\biggr)^{\frac{8}{3}} \frac{1}{G^{\frac{2}{3}}} \biggr) \ll \#\mathcal{R}_a U^2.
  \]
  In particular, the contribution to the right-hand side of \eqref{eq:Dsum_a} from those $a$ with
  $\#\mc R_a\ll R/W$ is
  \begin{equation}\label{eq:12}
    \ll \frac{1}{W}  H \biggl( \biggl(\frac{U\Delta}{\Delta}\biggr)^4G
    + \biggl(\frac{U\Delta}{\Delta}\biggr)^{\frac{4}{3}}G^{\frac{1}{3}}\biggr)\ll \frac{GU^4}{W}H
  \end{equation}
  If $W$ is some sufficiently large multiple of $GU^5$, \eqref{eq:12} is $\le \frac{1}{2}H/U\le \frac{1}{2}\#\mathcal{D}$.
  Taking such a choice of $W$ implies
  \[
    \#\mc D \ll \frac{H}{U} + \sum_{\substack{ \Delta G^{-\frac{1}{4}}U^{-\frac{3}{4}}\ll a\ll \Delta U\\ \#{\cal R}_a\ge R/W}}\#{\cal D}_a
    \ll \frac{H}{U} + U^2
    \sum_{\substack{ \Delta G^{-\frac{1}{4}}U^{-\frac{3}{4}}\ll a\ll \Delta U \\ \#{\cal R}_a\ge R/W}}\#{\cal R}_a.
  \]
  The bound \eqref{eq:lower_regime_est} then follows upon applying Proposition~\ref{prop:Ra_bound_lower} and summing over $a$.
\end{proof} 

The rest of this section is dedicated to the proof of Proposition~\ref{prop:Ra_bound_lower}.
Throughout this section, we suppose that
\begin{equation}\label{eq:cq6s1m0ltz}
   \frac{H}{\Delta^2}\gg GU^{10},  \qquad  \Delta \ggg G^2U^{5}.   
\end{equation}
Note that $\mathcal{R}_a$ is contained in an interval $I$ of length $\ll R$.
This, along with the fact that $\#\mathcal{R}_a\gg R/W$ implies that there exist  positive integers
$\ell_1 < \ell_2\ll W$ such that
\begin{equation}\label{eq:13}
  \#\mc R_a(\ell_1, \ell_2)\gg \#\mathcal{R}_a/W^2,
\end{equation}
where we write $\mc R_a(\ell_1, \ell_2) = \set{r: r, r + \ell_1, r + \ell_2\text{ consecutive elements of }\mc R_a }$.
Fix some such choice of $\ell_1, \ell_2$ from now on (and so dependencies on them may be omitted
when clear). 

Consider some $r\in\mc R_a(\ell_1, \ell_2)$. Write $d = d_a^*(r)$ and let $b_0 = b_0(r)$, $ v = v(r)\in \frac{1}{\ell_1} \mathbb{Z} $ be defined as
\begin{equation}\label{eq:10}
  b_0(r) := \frac{1}{\ell_1}(d_a^*(r + \ell_1) - d), \qquad v(r) := d_a^*(r + \ell_2) - (d + \ell_2b_0(r)).
\end{equation}
For clarity, let us get a handle on the scales involved here. First, note that
\begin{equation}\label{eq:b0_approx1} 
  b_0 = \frac{1}{\ell_1}(\tilde d_a(r + \ell_1) - \tilde d_a(r))
  + O\biggl(\frac{1}{\ell_1} \frac{\Delta}{G}\frac{\Delta^3}{a^3}\biggr) = \tilde b_a(r)
  + O\biggl(\Delta \frac{1}{\ell_1G\Omega^3}\biggr),
\end{equation}
where
\begin{equation}
  \tilde b_a(r) := \frac{1}{\ell_1}(\tilde d_a(r + \ell_1) - \tilde d_a(r)).
\end{equation}
By \eqref{eq:tilded_ests} , with
\begin{equation}\label{eq:1}
  B := \frac{D}{R}\asymp \frac{\Delta^2}{G}\pfrc{\Delta}{a}^3 = \Delta^2 \frac{1}{G\Omega^3},
\end{equation}
we have 
\begin{equation}
  \label{eq:tildeb_est}
  \tilde b_a^{(j)}(\rho)\asymp \frac{B}{R^j}
\end{equation}
for all $j\ge 0$. It may be helpful to record that \eqref{eq:b0_approx1}
is equivalent to 
\begin{equation}
  \label{eq:b0_err_mult}
  b_0 = \tilde b_a(r) \biggl(1 + O\pfrc{1}{\ell_1\Delta}\biggr).
\end{equation}
Furthermore, rearranging \eqref{eq:10}, applying \eqref{eq:b0_approx1} and
\eqref{eq:d_approx}, and then Taylor expanding $\tilde d_a$ yields that
\begin{multline}\label{eq:v_upper}
  v = d_a^*(r + \ell_2) - d_a^*(r + \ell_1) - (\ell_2 - \ell_1)b_0\\ 
  =  \tilde d_a(r+\ell_2)-\frac{\ell_2}{\ell_1}\tilde d_a(r+\ell_1)+\frac{\ell_2-\ell_1}{\ell_1}\tilde d_a(r)
  +  O \biggl( \Delta \frac{\ell_2 }{\ell_1G \Omega^3}\biggr)\\
  \ll \Delta\biggl(\frac{H}{R^2}\ell_2(\ell_2 - \ell_1) + \frac{\ell_2}{\ell_1G\Omega^3} \biggr)
  \ll  \Delta\biggl(\frac{\Delta^2U^{10}}{H\Omega^6} + \frac{U^5}{\Omega^3}\biggr)\ll \Delta \frac{U^5}{\Omega^3}.
\end{multline}
We note also the crude bound $v\lll B$. This will only be relevant in getting a handle
on terms in Taylor expansions.

We will show Proposition~\ref{prop:Ra_bound_lower} in this section with a new differencing identity which
implies restrictions on $v$ to more narrow ranges than \eqref{eq:v_upper}. With another simpler differencing
identity, we are reduced to counting problems which are dealt with using the results of \S\ref{sec:minor_lemmas}.

\subsection{Setup for some new differencings}\label{sec:cq6v97w1io}
In this subsection, we introduce some differencing identities.
Write 
\begin{equation}\label{eq:2}
  S_{a, b}(d) = -(b - a)\frac{X}{d^2} + (b + a)\frac{X}{(d + a)^2} - (b + a)\frac{X}{(d + b)^2} + (b - a)\frac{X}{(d + a + b)^2}.
\end{equation}
We remark that this differencing identity was introduced by Filaseta and Trifonov~\cite{MR1171549}
to show Lemma~\ref{lem:Da_space}. Our use of this identity will be quite different from its
original application.

From the Taylor expansion 
\[
  \frac{1}{(d + x)^2} = \frac{1}{d^2}\biggl(1 - 2\frac{x}{d} + 3\frac{x^2}{d^2} - \dots\biggr),
\]
it can be checked that
\begin{multline}\label{eq:S_taylor}
  S_{a, b}(d) = \frac{X}{d^5}\biggl(4(-ab^3 + a^3b)
  + 5\frac{2ab^4 + 2a^2b^3 - 2a^3b^2 - 2a^4b}{d}\\
  + 6\frac{-3ab^5 - 5a^2b^4 + 5a^4b^2 + 3a^5b}{d^2}
  + \dots\biggr).
\end{multline}
For technical reasons (slight issues with the $a^3b$ term above),
we will work with the following modified version of $S_{a, b}(d)$, defined as 
\begin{equation}\label{eq:6}
  \hat S_{a, b}(d) := S_{a, b}(d) - (R_a(d) - R_a(d + b)).
\end{equation}
It can be checked that
\begin{equation}\label{eq:Sh_exp}
  \hat S_{a, b}(d) = \frac{X}{d^5} \biggl(-4ab^3 + \frac{10ab^4 + 10a^2b^3}{d} -
  \frac{18ab^5 + 30a^2b^4 + 20a^3b^3}{d^2} + \dots\biggr).
\end{equation}
If $d, d + b\in{\cal D}_a$, then since
$X/d^2, X/(d+ a)^2, X/(d + b)^2, X/(d + b + a)^2\in \mathbb{Z} + O(1/(H\Delta^2))$ and
$\norm{R_a(d)}, \norm{R_a(d + b)}\ll 1/\Delta$, we have that
\begin{equation}\label{eq:15}
  \norm*{\hat S_{a, b}(d)} \ll \frac{WB}{H\Delta^2} + \frac{1}{\Delta}\ll \frac{1}{\Delta}, 
\end{equation}
where the last bound follows from \eqref{eq:cq6s1m0ltz}.
In this same situation, consider
\begin{equation}\label{eq:7}
  F_{a, b}(d) := F_a(d) - F_a(d + b).
\end{equation}
We have the expansion
\begin{equation} \label{eq:second_diff_exp}
  F_{a, b}(d) = \frac{X}{d^4}\biggl(6ab - 12\frac{a^2b + ab^2}{d}
  + 10\frac{2a^3b + 3a^2b^2 + 2ab^3}{d^2}
  + \dots\biggr).
\end{equation}
Before proceeding with the proof of Proposition~\ref{prop:Ra_bound_lower} in \S\ref{sec:cq6uvm8sc0}, \ref{sec:cq6vspz820}, we will 
outline the remainder of this section and make some remarks on what we just did.

In \S\ref{sec:cq6uvm8sc0}, we deal with the case of the defect $v$ being small.
To be precise, we deal with the case of $v\ll (\Delta^{1/2} + \Delta^3/H) (GU)^{O(1)}$ (though the case of
$v\lll \Delta$ also follows from the results of \S\ref{sec:cq6uvm8sc0}). Dealing with this case will not require the
elaborate new differencing $\Upsilon$ we will introduce in \S\ref{sec:cq6vspz820}.
We instead consider the difference $\Qoppa := \ell_1F_{a, \ell_2b_0 + v}(d) - \ell_2F_{a, \ell_1b_0}(d)  $ which has
been engineered to straightforwardly cancel the leading term in \eqref{eq:second_diff_exp} if $v = 0$ and is
particularly effective when $v$ is small.

In \S\ref{sec:cq6vspz820}, we deal with the case of $v$ being of ``generic'' size, near the upper bound
$O(\Delta (GU)^{O(1)})$ of \eqref{eq:v_upper}. It is here where we make use of $\Upsilon$, a differencing defined
in \S\ref{sec:cq6vspz820}, which has been engineered to remove the monomials
\begin{equation}
  \frac{Xab_0^3}{d^5}, \frac{Xab_0^2v}{d^5}, \frac{Xab_0^4}{d^6}
\end{equation}
from a linear combination of $\hat S_{a, b_0\ell_1}(d), \hat S_{a, b_0\ell_2 + v}(d), \hat S_{a, b_0(\ell_2 - \ell_1) + v}(d + \ell_1b_0)$.

It can be checked that we only have $2$ degrees of freedom, so it might seem we would only
be able to find coefficients that eliminate two of these. Remarkably, the coefficients
that eliminate the first two monomials also eliminate the third.

We do not have a good explanation for this minor miracle. A rough computation suggests
that a combination of the bracket polynomial methods of \S\ref{sec:upper}, \ref{sec:cq6v6y3nkr} along with the methods
of this section would be sufficient if such a miracle did not occur, but we have not
checked this carefully for we do not have an example where this is so.

\subsection{Small defects}\label{sec:cq6uvm8sc0}

The key differencing we will use in this section is
\begin{equation}\label{eq:9}
  \ZZ + O\pfrc{W}{H\Delta^2}\ni\Qoppa := \ell_1 F_{a, \ell_2b_0 + v}(d) - \ell_2 F_{a, \ell_1b_0}(d).
\end{equation}
Using \eqref{eq:second_diff_exp} along with \eqref{eq:1}, it can be seen that
\begin{equation}\label{eq:qoppa_eval}
  \Qoppa = 6\ell_1\frac{Xav}{d^4} - 12\ell_1\ell_2(\ell_2 - \ell_1)\frac{Xab_0^2}{d^5}
  + O \biggl( W^2\frac{Xab_0 v}{d^5} + W^3\frac{Xab_0^3}{d^6} \biggr).
\end{equation}
We record that
\begin{equation}
  \label{eq:qoppa_second_size}
  12\ell_1\ell_2(\ell_2 - \ell_1)\frac{Xab_0^2}{d^5}\asymp \ell_1\ell_2(\ell_2 - \ell_1)\frac{1}{G\Omega^5}
  \ll \ell_1\frac{GU^{10}}{\Omega^5}.
\end{equation}
We also record that the error terms in \eqref{eq:qoppa_eval} satisfy
\begin{align}\label{align:cq6yhxqmud}
  W^2 \frac{Xab_0 v}{d^5}\ll \frac{\Delta}{H} \frac{H|v|}{\Delta^3}\frac{G^2U^{10}}{\Omega^2}
  \ll \frac{1}{\Delta}\frac{G^2U^{15}}{\Omega^5},
  &&  W^3 \frac{Xab_0^3}{d^6}\asymp \frac{\Delta}{H} \frac{GU^{15}}{\Omega^8} 
\end{align}
In this section, we will show the following pair of results. The methods of this
subsection are capable of resolving all but the case of $v\gg \Delta(GU)^{-O(1)}$.
\begin{proposition}\label{prop:smallv}
  For any
  \begin{equation*}
    \frac{\Delta^3}{H} \ll V_+\ll \frac{\Delta^3}{H} \frac{U^{10}}{\Omega^6},
  \end{equation*}
  we have
  \begin{equation}\label{eq:cq6uvm6vi6}
    \#\set{r\in{\cal R}_a(\ell_1, \ell_2) : 0 \le |v(r)| \le V_+}\\  \ll
    \frac{H}{\Delta}\biggl(  \frac{\Delta^2}{H} \frac{G^4U^{35}}{\Omega^{8}} + \frac{1}{\Delta^{\frac{1}{2}}}
    \frac{G^{\frac{9}{2}}U^{25}}{\Omega^{2}} \biggr)
    \frac{H V_+}{\Delta^3}.
  \end{equation}
  Furthermore, there exists $C > 0$ such that for $V_+\ll \Delta U^5/\Omega^3 $, we have
  \begin{equation}\label{eq:cq6uvm8sc0}
    \#\set{r\in \mathcal{R}_a(\ell_1, \ell_2) :  C\frac{\Delta^3}{H} \frac{U^{10}}{\Omega^6} \le |v(r)| \le V_+} \ll \frac{H V_+}{\Delta^2} \frac{G^5U^{20}}{\Omega^2}.
  \end{equation}
\end{proposition}
\begin{proof}[Proof of Proposition~\ref{prop:smallv}]

  We begin with the proof of \eqref{eq:cq6uvm6vi6}.
  Consider some $r$ such that $|v(r)|\le V_+$. By \eqref{eq:qoppa_eval} and \eqref{eq:qoppa_second_size}, we have 
  \[
    |\Qoppa| \ll G^2U^5\Omega\pfrc{HV_+}{\Delta^3} + \frac{G^2U^{15}}{\Omega^5}.
  \]
  By \eqref{eq:d_approx}, \eqref{eq:cq6ufgf398}, and \eqref{eq:b0_err_mult}, we have
  \[
    \Qoppa = 6\ell_1\frac{Xa v}{\tilde d_a(r)^4}
    - 12\ell_1\ell_2(\ell_2 - \ell_1)\frac{Xa\tilde b_a(r)^2}{\tilde d_a(r)^5}
    + O\biggl(W^2 \frac{XaBV_+}{D^5} + \frac{1}{\Delta}\frac{G^2U^{15}}{\Omega^5}\biggr),
  \]
  where the new error term is ultimately due to \eqref{eq:b0_err_mult} from the approximation
  of $b_0$ in the second term. The above error can be simplified somewhat crudely with
  \eqref{eq:cq6ufgf398} so that we have 
  \begin{equation}
    \Qoppa = 6\ell_1\frac{Xa v}{\tilde d_a(r)^4}
    - 12\ell_1\ell_2(\ell_2 - \ell_1)\frac{Xa\tilde b_a(r)^2}{\tilde d_a(r)^5}
    + O\biggl(\frac{1}{\Delta}\frac{G^2U^{20}}{\Omega^5}\biggr).
\end{equation}
  That $\norm{\Qoppa}\ll \frac{W}{H\Delta^2}$ implies that there exists some integer 
  \begin{equation}\label{eq:14}
    f \ll  G^2U^5\Omega \frac{H V_+}{\Delta^3} +  \frac{G^2U^{15}}{\Omega^5}
    \ll \frac{G^2U^{15}}{\Omega^5} 
  \end{equation}
  such that
  \[
    6\ell_1\frac{Xa v}{\tilde d_a(r)^4} =
    f + 12\ell_1\ell_2(\ell_2 - \ell_1)\frac{Xa\tilde b_a(r)^2}{\tilde d_a(r)^5}
    + O\biggl(\frac{1}{\Delta}\frac{G^2U^{20}}{\Omega^5}\biggr),
  \]
  Multiplying by $\frac{\tilde d_a(r)^4}{6\ell_1Xa}$ to solve for $v$, we obtain that
  \[
    v = \frac{\tilde d_a(r)^4}{6\ell_1Xa}\biggl(f
    + 12\ell_1\ell_2(\ell_2 - \ell_1)\frac{Xa\tilde b_a(r)^2}{\tilde d_a(r)^5}\biggr)
    + O\biggl(\frac{\Delta^2}{H}\frac{1}{\ell_1} \frac{GU^{20}}{\Omega^6}\biggr).
  \]
  Since $v\in \frac{1}{\ell_1}\mathbb{Z}$, with
  \[
    \phi_f(r) = \frac{\tilde d_a(r)^4}{6Xa}
    \biggl(f + 12\ell_1\ell_2(\ell_2 - \ell_1)\frac{Xa\tilde b_a(r)^2}{\tilde d_a(r)^5}\biggr),
  \]
  we have
  \[
    \norm{\phi_f(r)} \ll \frac{\Delta^2}{H} \frac{GU^{20}}{\Omega^6}.
  \]
  It follows that
  \begin{equation}
    \#\set{r\in{\cal R}_a(\ell_1, \ell_2) :  |v(r)|\le V_+} \ll
    \sum_{|f|\ll \frac{G^2U^{15}}{\Omega^5} }
    \sum_{r\asymp R} \charf{\norm{\phi_f(r)}\ll \frac{\Delta^2}{H}\frac{GU^{20}}{\Omega^6} }.
  \end{equation}
  With
  \[
    T = \frac{\Delta^3}{H}\max \biggl( \frac{|f|}{G\Omega}, \ell_1\ell_2(\ell_2 - \ell_1) \frac{1}{G^2\Omega^6} \biggr),
  \]
  it can be checked that for $\rho\asymp R$, we have
  \[
    |\phi_f'(\rho)| + R|\phi_f''(\rho)|\asymp \frac{T}{R}
  \]
  (since up to lower order terms, $\phi_f$ is the sum of two monomials).
  Crudely, we have that
  \begin{equation}\label{eq:cq6wp32tny}
    \frac{\Delta^3}{H} \frac{1}{G^2\Omega^6} \ll T \ll \frac{\Delta^3}{H} \frac{GU^{15}}{\Omega^6}.
  \end{equation}
  Therefore, the conditions of Lemma~\ref{lem:crude_stat_est} are satisfied
  (with $T, X, \delta, \phi$ taken to be $T, R, \frac{\Delta^2}{H}\frac{GU^{20}}{\Omega^6}, \phi_f$ respectively)
  and so we have (using the bounds \eqref{eq:cq6wp32tny} and \eqref{eq:cq6s1m0ltz} to simplify)
  \begin{multline}
    \sum_{r\asymp R} \charf{\norm{\phi_f(r)}\ll
      \frac{\Delta^2}{H}\frac{GU^{20}}{\Omega^6}} 
    \ll R \biggl(  \frac{\Delta^2}{H}\frac{GU^{20}}{\Omega^6}
    + \frac{\frac{\Delta}{H^{1/2}}\frac{G^{1/2}U^{10}}{\Omega^3}}{\frac{\Delta^{3/2}}{H^{1/2}}\frac{1}{G\Omega^3}}\biggr)
    + T + 1\\
    \ll  \frac{H}{\Delta} G\Omega^3 \biggl(\frac{\Delta^2}{H}\frac{GU^{20}}{\Omega^6}
    + \frac{1}{\Delta^{\frac{1}{2}}}G^{\frac{3}{2}}U^{10}\biggr)
    + \frac{\Delta^3}{H} \frac{GU^{15}}{\Omega^6}+ 1 \\ 
    \ll \frac{H}{\Delta}\biggl(  \frac{\Delta^2}{H} \frac{G^2U^{20}}{\Omega^3}
    + \frac{1}{\Delta^{\frac{1}{2}}}G^{\frac{5}{2}}U^{10}\Omega^3 \biggr).
  \end{multline}
  Then, \eqref{eq:cq6uvm6vi6} follows after summing over $|f|\ll \frac{G^2U^{15}}{\Omega^5}$ (recalling \eqref{eq:14}). 

  It now remains to show \eqref{eq:cq6uvm8sc0}. Suppose now that $|v(r)|\ggg  (\Delta^3/H)(U^{10}/\Omega^6)$.
  Then, we have that
  \begin{equation}
    |\Qoppa| \asymp \ell_1G\Omega \frac{H |v|}{\Delta^3}.
  \end{equation}
  We will bound the contribution of $|v(r)| \sim V$ for
  \begin{equation}
    C \frac{\Delta^3}{H} \frac{U^{10}}{\Omega^6} \le  V\ll V_+
  \end{equation}
  with $C$ sufficiently large. We shall also suppose that
  \begin{equation}
    V_+ \ll \Delta \frac{U^5}{\Omega^3}
  \end{equation}
  in light of the bound \eqref{eq:v_upper}. 
  Then, with $f$ as before, we have that
  \begin{equation}
    |f|\asymp \ell_1G\Omega \frac{H V}{\Delta^3}.
  \end{equation}
  
  Also, with $\phi_f$ as before, we now have that for $C$ sufficiently large\footnote{To be clear, the implied constants will not depend on the choice of $C$, so long as it is sufficiently large}, $|\phi_f'(\rho)|\asymp T/R$ for $\rho\asymp R$, now with
  \begin{equation}
    \frac{\Delta^3}{H} \ll T = \frac{D^4}{Xa}|f| \ll \frac{\Delta^3}{H} GU^5 \frac{HV}{\Delta^3}.
  \end{equation}
  Therefore, by Lemma \ref{lem:crude_frac_bound}, since $T\gg \frac{\Delta^2}{H}\frac{GU^{15}}{\Omega^5}$, the contribution to \eqref{eq:cq6uvm6vi6}
  of $|v(r)|\sim V \ll \Delta U^5/\Omega^3$ (which we may suppose by \eqref{eq:v_upper}) is at most
  \begin{multline}
    \sum_{r\asymp R} \charf{\norm{\phi_f(r)}\ll \frac{\Delta^2}{H}\frac{GU^{15}}{\Omega^6}}
    \ll R \frac{\Delta^2}{H}\frac{GU^{15}}{\Omega^6} + T + 1
    \ll \frac{H}{\Delta} \biggl( \frac{\Delta^2}{H} \frac{G^2U^{15}}{\Omega^3}
    + \frac{\Delta^4}{H^2} GU^5 \frac{H V}{\Delta^3} + \frac{\Delta}{H} \biggr)\\
    \ll \frac{H}{\Delta} \frac{\Delta^2}{H}\biggl( \frac{G^2U^{15}}{\Omega^3} + \frac{G^3U^{10}}{\Omega^3} \biggr)
    \ll \frac{H}{\Delta} \frac{\Delta^2}{H} \frac{G^3U^{15}}{\Omega^3}.
  \end{multline}
  Summing over scales $V\ll V_+$ and $|f|\asymp G^2U^5\Omega \frac{H V}{\Delta^3}$ yields that

  \begin{equation}
    \#\set{r\in \mathcal{R}_a(\ell_1, \ell_2) :  C\frac{\Delta^3}{H} \frac{U^{10}}{\Omega^6} \le |v(r)| \le V_+}
    \ll \frac{H}{\Delta} \frac{\Delta^2}{H} \frac{G^3U^{15}}{\Omega^3} G^2U^5\Omega \frac{H V}{\Delta^3}
    \ll \frac{H V_+}{\Delta^2} \frac{G^5U^{20}}{\Omega^2},
  \end{equation}
  as desired.
\end{proof}

\subsection{Large defects}\label{sec:cq6vspz820}

In this section, we will show bounds on the size of the rest of $\#\mathcal{R}_a(\ell_1, \ell_2)$ by showing the following
proposition.
\begin{proposition}\label{proposition:cq6v970jxq}
  There exists $C > 0$ such that
  \begin{multline}
    \#\set{r\in \mathcal{R}_a(\ell_1, \ell_2)  : |v(r)|\ge
      C \biggl( \frac{\Delta^3}{H}\frac{G^{\frac{5}{2}}U^{\frac{45}{2}}}{\Omega^6}
      + \Delta^{\frac{1}{2}}G^2U^{10}\Omega \biggr)}\\ 
    \ll \frac{H}{\Delta}\biggl(\frac{1}{\Delta}\frac{G^{14}U^{75}}{\Omega^{13}}
    + \frac{\Delta^2}{H} \frac{G^{14}U^{90}}{\Omega^{27}}\biggr).
\end{multline}
\end{proposition}

The range of $v $ in this section will imply that $|p_1(v)| \ggg |p_2(v)|/d $. 

Before we prove Proposition~\ref{proposition:cq6v970jxq}, we will introduce one final
new differencing.
Recall the situation prior to \S\ref{sec:cq6v97w1io}: we were considering
$r\in {\cal R}_a(\ell_1, \ell_2), d = d_a^*(r)\in {\cal D}_a$ and $b_0, v\in \frac{1}{\ell_1}\ZZ$ such that $d + \ell_1b_0, d + \ell_2 b_0 + v\in{\cal D}_a$.

From \eqref{eq:Sh_exp}, we have that
\begin{multline}
  \label{eq:Shtrunc_pre}
  \hat S_{a, b}(d) = \frac{X}{d^5}\biggl(\biggl(-4a + 10\frac{a^2}{d}\biggr)b^3
  + 10\frac{ab^4}{d}\biggr) +
  O\biggl(\frac{Xab^5}{d^7}\biggr)\\
  = \frac{X}{d^5}\biggl(-4a + 10\frac{a^2}{d}\biggr)\biggl(b^3 - \frac{5}{2}\cdot \frac{b^4}{d}\biggr) + O\biggl(\frac{Xab^5}{d^7}\biggr). 
\end{multline}
For $b\in\set{\ell_1b_0, \ell_2 b_0, (\ell_2 - \ell_1)b_0, \ell_2b_0 + v, (\ell_2 - \ell_1)b_0 + v}$, we have that
$b\ll WB \asymp \Delta^2 U^5/\Omega^3$, and so the remainder in \eqref{eq:Shtrunc_pre} satisfies
\begin{equation}
  \label{eq:erroreval_trunc}
  \frac{Xab^5}{d^7} \ll \frac{\Delta^4}{H^2} \frac{GU^{25}}{\Omega^{14}}.
\end{equation}
Now, consider the quantity
\begin{equation}\label{eq:8}
  \Upsilon := \ell_2^2(\ell_2 - \ell_1)^2\hat S_{a, b_0\ell_1}(d)
  - \ell_1^2(\ell_2 - \ell_1)^2\hat S_{a, b_0\ell_2 + v}(d)
  + \ell_1^2\ell_2^2\hat S_{a, b_0(\ell_2 - \ell_1) + v}(d + \ell_1b_0).
\end{equation}
We have that $\norm{\Upsilon}\ll W^4/\Delta $ since $\norm*{\hat S_{a, b_0\ell_1}(d)}\ll 1/\Delta$, recalling \eqref{eq:15}.
Also, from \eqref{eq:Shtrunc_pre} and the expansion
\[
  \frac{X}{(d + \ell_1b_0)^5} = \frac{X}{d^5}\biggl(1 - 5\frac{\ell_1b_0}{d} + O\pfrc{W^2b_0^2}{d^2}\biggr),
\]
we have that
\begin{multline}\label{eq:Ups_exp}
  \hat S_{a, b_0(\ell_2 - \ell_1) + v}(d + \ell_1b_0)\\
  = \frac{X}{(d + \ell_1b_0)^5}\biggl(-4a + 10\frac{a^2}{d}\biggr)
  \cdot\biggl((b_0(\ell_2 - \ell_1) + v)^3 - \frac{5}{2}
  \cdot \frac{((\ell_2 - \ell_1)b_0 + v)^4}{d}\biggr)
  + O \biggl( \frac{\Delta^4}{H^2} \frac{GU^{25}}{\Omega^{14}} \biggr)\\
  = \frac{X}{d^5}\biggl(-4a + 10\frac{a^2}{d}\biggr)\biggl((b_0(\ell_2 - \ell_1) + v)^3 -
  \frac{5}{2}\cdot \frac{((\ell_2 - \ell_1)b_0 + v)^4 + 2\ell_1b_0((\ell_2 - \ell_1)b_0 + v)^3}{d}\biggr)\\
  + O\biggl( \frac{\Delta^4}{H^2} \frac{GU^{25}}{\Omega^{14}}\biggr).
\end{multline}
Then, noting that
\begin{multline}\label{multline:cq5rr17ch5}
  \ell_2^2(\ell_2 - \ell_1)^2 (b_0\ell_1)^3
  - \ell_1^2(\ell_2 - \ell_1)^2(b_0\ell_2 + v)^3
  + \ell_1^2\ell_2^2(b_0(\ell_2 - \ell_1) + v)^3\\
  = 3\ell_1^3\ell_2(\ell_2 - \ell_1)b_0v^2 + \ell_1^3(2\ell_2 - \ell_1)v^3
  =: p_1(v)
\end{multline}
and
\begin{multline}\label{multline:cq5rr2avlb}
  -\frac{5}{2}(\ell_2^2(\ell_2 - \ell_1)^2 (b_0\ell_1)^4
  - \ell_1^2(\ell_2 - \ell_1)^2(b_0\ell_2 + v)^4
  + \ell_1^2\ell_2^2((b_0(\ell_2 - \ell_1) + v)^4 + 2\ell_1b_0((\ell_2 - \ell_1)b_0 + v)^3))\\
  = -5\ell_1^3\ell_2^2(\ell_2 - \ell_1)^2b_0^3v - 15\ell_1^3\ell_2^2(\ell_2 - \ell_1)b_0^2v^2
  - 5\ell_1^3\ell_2(3\ell_2 - 2\ell_1)b_0v^3 - \frac{5}{2}\ell_1^3(2\ell_2 - \ell_1)v^4
  =: p_2(v),
\end{multline}
we obtain that
\begin{equation}\label{eq:Ups_eval}
  \Upsilon = \frac{X}{d^5}\biggl(-4a + 10\frac{a^2}{d}\biggr)\biggl(p_1(v) + \frac{p_2(v)}{d}\biggr) +
  O \biggl( \frac{\Delta^4}{H^2} \frac{G^5U^{45}}{\Omega^{14}}\biggr).
\end{equation}

By \eqref{eq:1}, \eqref{eq:v_upper}, and \eqref{eq:cq6s1m0ltz}, we have $|b_0|\ggg |v| $. Therefore, $p_1(v)$ and $p_2(v)$ are, in the
ranges we are concerned, dominated in absolute value by the terms $3\ell_1^3\ell_2(\ell_2 - \ell_1)b_0v^2$, $5\ell_1^3\ell_2^2(\ell_2 - \ell_1)^2b_0^3v$
respectively.
The range of $v$ in which $|p_1(v)|\ll |{p_2(v)}/d|$ was disposed of in \S\ref{sec:cq6uvm8sc0} (without appeal to
the differencing $\Upsilon$). In this section, we will be dealing with the remaining cases.

\begin{proof}[Proof of Proposition \ref{proposition:cq6v970jxq}]
  Fix some choice of $r\in \mathcal{R}_a(\ell_1, \ell_2)$ satisfying the constraints on elements of the set
  in Proposition~\ref{proposition:cq6v970jxq}.
  Note that the lower bound on $v(r)$ implies that
  \begin{equation}\label{eq:cq6v94rqzi}
    \biggl|\frac{p_1(v)}{p_2(v)/d}\biggr| \asymp \frac{H\Delta}{\ell_2(\ell_2 - \ell_1) |b_0|^2}|v|
    \gg \frac{1}{G^2U^{10}} G^2\Omega^6 \frac{H v}{\Delta^3}\gg C.
  \end{equation}
  Therefore, for $C$ sufficiently large, by \eqref{eq:Ups_eval} and the bound $v\ll \Delta U^5/\Omega^3$
  from \eqref{eq:v_upper}, we have that 
  \begin{equation}\label{multline:cq6utntrin}
    |\Upsilon| \asymp X\frac{a}{d^5}|p_1(v)| + O \biggl( \frac{\Delta^4}{H^2}\frac{G^5U^{45}}{\Omega^{14}} \biggr) 
    \asymp \frac{Xa}{d^5}\ell_1^3\ell_2 (\ell_2 - \ell_1) |b_0|v^2
    + O\biggl(\frac{\Delta^4}{H^2}\frac{G^5U^{45}}{\Omega^{14}}\biggr).
  \end{equation}
  where we have crudely combined \eqref{eq:cq6ufgf398}, \eqref{eq:cq6s1m0ltz} for the final bound.
  Also, 
  \begin{equation}\label{eq:cq6yhukdfb}
    \frac{X a}{d^5}\ell_1^3\ell_2(\ell_2 - \ell_1)b_0 v^2\gg
    C^2\max \biggl( \frac{\Delta^4}{H^2}\frac{G^5U^{45}}{\Omega^{14}}, \frac{1}{\Delta}G^4U^{20}\biggr).
  \end{equation}
  This along with the fact that $\norm{\Upsilon}\ll W^4/\Delta\asymp G^4U^{20}/\Delta$ implies that there exists a
  nonzero integer $s$ with
  \begin{equation}\label{eq:17}
    \Upsilon = s + O \biggl( \frac{G^4U^{20}}{\Delta} \biggr). 
  \end{equation}
  In particular, 
  \begin{equation}
    \frac{X}{d^5} \biggl( -4a + 10 \frac{a^2}{d} \biggr)
    \biggl( p_1(v) + \frac{p_2(v)}{d} \biggr)
    = s + O \biggl( \frac{\Delta^4}{H^2} \frac{G^5U^{45}}{\Omega^{14}} + \frac{1}{\Delta} G^4U^{20} \biggr).
  \end{equation}
  We also record that by \eqref{multline:cq6utntrin} and the fact that $s$ is nonzero for $C$ sufficiently large
  (which follows from \eqref{multline:cq6utntrin}, \eqref{eq:cq6yhukdfb}, and \eqref{eq:17}), we have
  \begin{equation}\label{eq:cq6v988b6h}
    1\le |s|\ll \frac{Xa}{D^5}\ell_1^3\ell_2(\ell_2 - \ell_1)B \biggl(\Delta\frac{U^5}{\Omega^3}\biggr)^2
    \asymp \ell_1^3\ell_2(\ell_2 - \ell_1)\frac{U^{10}}{\Omega^8}\ll \frac{G^5U^{35}}{\Omega^8}.
  \end{equation}
  Now, let
  \begin{align*}
    \tilde p_1(v, r) &= 3\ell_1^3\ell_2(\ell_2 - \ell_1)\tilde b_a(r)v^2 + \ell_1^3(2\ell_2 - \ell_1)v^3,\\
    \tilde p_2(v, r) &= -5\ell_1^3\ell_2^2(\ell_2 - \ell_1)^2\tilde b_a(r)^3v
                    - 15\ell_1^3\ell_2^2(\ell_2 - \ell_1)\tilde b_a(r)^2v^2\\
                  &\hspace{1cm}- 5\ell_1^3\ell_2(3\ell_2 - 2\ell_1)\tilde b_a(r)v^3
                    - \frac{5}{2}\ell_1^3(2\ell_2 - \ell_1)v^4.
  \end{align*}
  By our approximations \eqref{eq:d_approx}, \eqref{eq:b0_err_mult} for $d, b_0$ respectively, we obtain that
  \begin{multline}\label{eq:cq6v987zct}
    \Sampi(r) := \frac{X}{\tilde d_a(r)^5}\biggl(-4a + 10\frac{a^2}{\tilde d_a(r)}\biggr)\biggl(\tilde p_1(v, r)
    + \frac{\tilde p_2(v, r)}{\tilde d_a(r)}\biggr) \\ 
    = s + O \biggl( \frac{\Delta^4}{H^2} \frac{G^5U^{45}}{\Omega^{14}}
    + \frac{1}{\Delta} G^4U^{20} + \frac{\ell_1^2\ell_2(\ell_2 - \ell_1)v^2}{\Delta^3\Omega^2} \biggr).
  \end{multline}
  We do not solve \eqref{eq:cq6v987zct} for $r$. Instead, we use it to constrain the possible
  values of $v$.
  Using $B\asymp \Delta^2/(G\Omega^3)$, $D=H\Delta$, the crude bound $|v|\lll B$, and \eqref{eq:cq6s1m0ltz}, expanding out yields
  that 
  \begin{multline}\label{eq:sampi_parabola_explicit}
    \Sampi(r)
    =
    \frac{3\ell_1^3\ell_2(\ell_2-\ell_1)}{a^2}v^2 \\
    + O\biggl(
    \frac{\ell_1^2\ell_2(\ell_2-\ell_1)v^2}{\Delta^3\Omega^2}
    + \frac{\ell_1^3\ell_2^2(\ell_2-\ell_1)v^2}{HG\Delta\Omega^5}
    + \frac{G\Omega}{\Delta^4}\ell_1^3\ell_2|v|^3
    + \frac{\ell_1^3\ell_2^2(\ell_2-\ell_1)^2\Delta}{HG^2\Omega^8}|v|
    \biggr).
  \end{multline}
  Combining \eqref{eq:sampi_parabola_explicit} with \eqref{eq:cq6v987zct}, and using the
  large-defect lower bound on $|v|$, gives
  \begin{equation}\label{eq:cq6waeioej}
    |v|\asymp
    \biggl(\frac{D^5|s|}{XAB\ell_1^3\ell_2(\ell_2-\ell_1)}\biggr)^{1/2}
    =: V_s.
  \end{equation}
  In particular,
  \[
    V_s
    =
    \Delta\Omega
    \biggl(\frac{|s|}{\ell_1^3\ell_2(\ell_2-\ell_1)}\biggr)^{1/2}.
  \]
  For this fixed $s$, put
  \[
    N_s
    =
    \#\set{v\in \ell_1^{-1}\ZZ:\ \exists r\in \mathcal{R}_a(\ell_1,\ell_2)
    \text{ such that \eqref{eq:cq6v987zct} holds with this }s}.
  \]
  Subtracting two instances of \eqref{eq:sampi_parabola_explicit} and applying the
  square-difference count gives
  \begin{multline}\label{eq:fixed_s_v_count_explicit}
    N_s
    \ll 1
    + \frac{\Delta\Omega
    \bigl(\frac{G^4U^{20}}{\Delta}
    +\frac{\Delta^4}{H^2}\frac{G^5U^{45}}{\Omega^{14}}\bigr)}
    {(\ell_1\ell_2(\ell_2-\ell_1)|s|)^{1/2}}\\
    + \frac{|s|^{1/2}}{(\ell_1\ell_2(\ell_2-\ell_1))^{1/2}}
    \biggl(
    \frac{\Omega}{\ell_1}
    +\frac{\ell_2\Delta^2}{HG\Omega^2}
    \biggr)
    + \frac{G\Omega^5\ell_2}{(\ell_1\ell_2(\ell_2-\ell_1))^2}|s|
    +\frac{\ell_1\ell_2(\ell_2-\ell_1)\Delta^3}{HG^2\Omega^6}.
  \end{multline}

  Now, recalling \eqref{eq:qoppa_eval} and \eqref{align:cq6yhxqmud}, consider
  \begin{equation}
    \varphi_{v}(r)
    :=
    6\ell_1 \frac{Xav}{\tilde d_a(r)^4}
    - 12\ell_1\ell_2(\ell_2 - \ell_1)\frac{Xa\tilde b_a(r)^2}{\tilde d_a(r)^5}
    =
    \Qoppa + O \biggl( \frac{1}{\Delta} \frac{G^4U^{15}}{\Omega^5} \biggr).
  \end{equation}
  By \eqref{eq:cq6waeioej}, for $C$ sufficiently large,
  \[
    |\varphi_v'(\rho)|\asymp \frac{T}{R},
    \qquad
    T=\ell_1\frac{XAV_s}{D^4},
  \]
  and hence
  \[
    \frac{T}{R}\asymp
    \frac{1}{\Delta\Omega}
    \biggl(\frac{|s|}{\ell_1\ell_2(\ell_2-\ell_1)}\biggr)^{1/2}.
  \]
  Write
  \[
    \delta = \frac{1}{\Delta}\frac{G^4U^{15}}{\Omega^5}.
  \]
  Lemma~\ref{lem:crude_frac_bound} gives for each fixed $v$ occurring with this $s$ that
  \begin{equation}\label{eq:fixed_v_r_count_large_defect}
    \#\set{r:\norm{\varphi_v(r)}\ll \delta}
    \ll
    R\delta+\frac{\delta}{T/R}
    \asymp
    \frac{H}{\Delta^2}\frac{G^5U^{15}}{\Omega^2}
    +
    \frac{G^4U^{15}}{\Omega^4}
    \biggl(\frac{\ell_1\ell_2(\ell_2-\ell_1)}{|s|}\biggr)^{1/2}.
  \end{equation}
  Therefore
  \begin{multline}
    \#\set{r\in \mathcal{R}_a(\ell_1,\ell_2): |v(r)|\ge
      C \biggl( \frac{\Delta^3}{H}\frac{G^{\frac{5}{2}}U^{\frac{45}{2}}}{\Omega^6}
      + \Delta^{\frac{1}{2}}G^2U^{10}\Omega \biggr)}\\
    \ll
    \sum_{1\le |s|\ll
    \ell_1^3\ell_2(\ell_2-\ell_1)U^{10}/\Omega^8}
    N_s
    \cdot
    \biggl(
    \frac{H}{\Delta^2}\frac{G^5U^{15}}{\Omega^2}
    +
    \frac{G^4U^{15}}{\Omega^4}
    \biggl(\frac{\ell_1\ell_2(\ell_2-\ell_1)}{|s|}\biggr)^{1/2}
    \biggr).
  \end{multline}
  Substituting \eqref{eq:fixed_s_v_count_explicit}, using
  \[
    \sum_{1\le |s|\ll
    \ell_1^3\ell_2(\ell_2-\ell_1)U^{10}/\Omega^8}1
    \ll
    \ell_1^3\ell_2(\ell_2-\ell_1)\frac{U^{10}}{\Omega^8},
  \]
  \[
    \sum_{1\le |s|\ll
    \ell_1^3\ell_2(\ell_2-\ell_1)U^{10}/\Omega^8}|s|^{-1/2}
    \ll
    \bigl(\ell_1^3\ell_2(\ell_2-\ell_1)\bigr)^{1/2}
    \frac{U^5}{\Omega^4},
  \]
  \[
    \sum_{1\le |s|\ll
    \ell_1^3\ell_2(\ell_2-\ell_1)U^{10}/\Omega^8}|s|^{1/2}
    \ll
    \bigl(\ell_1^3\ell_2(\ell_2-\ell_1)\bigr)^{3/2}
    \frac{U^{15}}{\Omega^{12}},
  \]
  \[
    \sum_{1\le |s|\ll
    \ell_1^3\ell_2(\ell_2-\ell_1)U^{10}/\Omega^8}|s|
    \ll
    \bigl(\ell_1^3\ell_2(\ell_2-\ell_1)\bigr)^2
    \frac{U^{20}}{\Omega^{16}},
  \]
  and the trivial bound for the $|s|^{-1}$-sum by the first displayed sum, and then applying
  $\ell_1,\ell_2\ll GU^5$, $\Omega\ll U$, $1\ll GU^3\Omega^4$,
  $G^2U^5\ll\Delta$, and $GU^{10}\ll H/\Delta^2$, the preceding display is
  \[
    \ll
    \frac{H}{\Delta}
    \biggl(
    \frac{1}{\Delta}\frac{G^{14}U^{75}}{\Omega^{13}}
    +
    \frac{\Delta^2}{H}\frac{G^{14}U^{90}}{\Omega^{27}}
    \biggr),
  \]
  as desired.
\end{proof}
\subsection{Putting together the ranges}

In this section, we apply Propositions \ref{prop:smallv} and \ref{proposition:cq6v970jxq} in their various ranges
to obtain Proposition \ref{prop:Ra_bound_lower}. We begin with an upper bound for $\mathcal{R}_a(\ell_1, \ell_2)$ for any $\ell_1 < \ell_2 \ll W$.
Applying \eqref{eq:cq6uvm6vi6} with $V_+ $ a sufficiently large multiple of $\frac{\Delta^3}{H}\frac{U^{10}}{\Omega^6}$, \eqref{eq:cq6uvm8sc0} with $V_+ $ a sufficiently
large multiple of $ \min(\frac{\Delta^3}{H}\frac{G^{\frac{5}{2}}U^{\frac{45}{2}}}{\Omega^6} + \Delta^{\frac{1}{2}}G^2U^{10}\Omega, \Delta U^5/\Omega^3)$, and
Proposition \ref{proposition:cq6v970jxq}, we obtain that $ \#\mathcal{R}_a(\ell_1, \ell_2)$ is $\ll$
\begin{multline}
  \frac{H}{\Delta} \biggl( \frac{1}{\Delta} \frac{G^4U^{10}}{\Omega^2}
  + \frac{\Delta^2}{H} \frac{G^4U^{45}}{\Omega^{14}}
  + \frac{1}{\Delta^{\frac{1}{2}}}\frac{G^{\frac{9}{2}}U^{35}}{\Omega^{8}}
  + \frac{\Delta^2}{H} \frac{G^{\frac{15}{2}}U^{\frac{85}{2}}}{\Omega^8}
  + \frac{1}{\Delta^{\frac{1}{2}}}\frac{G^7U^{30}}{\Omega}
  + \frac{1}{\Delta}\frac{G^{14}U^{75}}{\Omega^{13}} + \frac{\Delta^2}{H} \frac{G^{14}U^{90}}{\Omega^{27}}\biggr)\\
  \ll \frac{H}{\Delta} \biggl( \frac{1}{\Delta^{\frac{1}{2}}} \frac{G^7U^{41}}{\Omega}
  + \frac{\Delta^2}{H} \frac{G^{14}U^{90}}{\Omega^{27}} + \frac{1}{\Delta}\frac{G^{14}U^{75}}{\Omega^{13}}\biggr).
\end{multline}
Recalling that $\# \mathcal{R}_a \ll W^2 \mathcal{R}_a(\ell_1, \ell_2) $ and $W^2\asymp G^2U^{10} $, we obtain
\begin{equation}
  \#\mathcal{R}_a\ll \frac{H}{\Delta} \biggl( \frac{1}{\Delta^{\frac{1}{2}}} \frac{G^9U^{51}}{\Omega}
  + \frac{\Delta^2}{H} \frac{G^{16}U^{100}}{\Omega^{27}} + \frac{1}{\Delta} \frac{G^{16}U^{85}}{\Omega^{13}}\biggr),
\end{equation}
as desired.

\section{Proof of estimates for $\Delta$ around and above $H^{\frac 12}$} \label{sec:upper}

The main bounds of this section are:

\begin{proposition}\label{prop:upper_regime_est_easy}
  For $\Delta G^{-\frac{1}{4}} U^{-\frac{3}{4}}\ll A\ll \Delta U $, we have that
  \[
    \sum_{a\sim A}\#\mc R_a\ll H\biggl(\frac{H}{\Delta^2} G\Omega^2
    + \frac{H^{\frac 12}}{\Delta} G^{\frac{3}{2}}\Omega^{\frac{9}{2}} + \frac{1}{\sqrt{H}}
    G^{\frac{1}{2}}\Omega^{\frac{5}{2}}\biggr),
  \]
  where we write $\Omega = A/\Delta$.
\end{proposition}
\begin{proposition}\label{prop:upper_regime_est_nontrivial}
  There exists $\sigma > 0$ such that if $G, U\ll X^{\frac{1}{10^5}}$ and  
  \[
    X^{-\frac{1}{10^5}} \ll \frac{H}{\Delta^2}\ll X^{\frac{1}{10^5}},
  \]
  then for $a $ satisfying $\Delta G^{-\frac{1}{4}} U^{-\frac{3}{4}}\ll  a \ll \Delta U $, we have that 
  \[
    \#\mc R_a\ll \frac{H}{\Delta} \biggl(1 + \frac{H}{\Delta^2}\biggr)X^{-\sigma} (GU)^{O(1)}.
  \]
\end{proposition}

That Propositions \ref{prop:upper_regime_est_easy}, \ref{prop:upper_regime_est_nontrivial} imply \eqref{eq:upper_regime_easy}, \eqref{eq:upper_regime_est} respectively follows
from the fact that by Proposition \ref{prop:approx_param} and \eqref{eq:cq6ufgf398},
\begin{equation}
  \#\mathcal{D} \ll \frac{H}{U} +  \sum_{\Delta G^{-\frac{1}{4}}U^{-\frac{3}{4}} \ll a\ll \Delta U} \#\mathcal{D}_a
  \ll \frac{H}{U} + \sum_{\Delta G^{-\frac{1}{4}}U^{-\frac{3}{4}}\ll A\ll \Delta U \\ A = 2^k }
  \biggl(1 + G^{-\frac{2}{3}} \biggl(\frac{A}{\Delta}\biggr)^{-\frac{8}{3}}\biggr)
   \sum_{a\sim A}\#\mathcal{R}_a.
\end{equation}

The main input required to show these bounds, apart from manipulations involving
the implications of our ``approximate parametrization'' Proposition~\ref{prop:approx_param},
will be the results of \S\ref{sec:bracket_ests} which ultimately come from Green and Tao's
\cite{GT} quantitative version of Leibman's theorem on the equidistribution of
polynomial sequences on nilmanifolds.

\subsection{Further implications of the approximate parametrization and a
  proof of Proposition \ref{prop:upper_regime_est_easy}}

Let $f_a^*(r)$ be the integer nearest $F_a(d_a^*(r))$, and let
$\tilde f_a(r) = F_a(\tilde d_a(r))$. By \eqref{eq:std_exp},
\[
  F_a(d) \asymp \frac{Xa}{d^3} \asymp \frac{H^2}{\Delta^2} G\Omega =: F 
\]

Also, by \eqref{eq:d_approx}, we have
\begin{equation}\label{eq:f_approx} 
  \tilde f_a(r) = f_a^*(r) + O\biggl(F\frac{\Delta^3H^4}{XA^3}\biggr)
  = f_a^*(r) + O\biggl(\frac{H}{\Delta^2}\frac{1}{\Omega^2}\biggr).
\end{equation}

We therefore have that
\[
  \#{\cal R}_a \ll \sum_{r\asymp R} \charf{\norm*{\tilde f_a(r)}\ll \frac{H}{\Delta^2}\frac{1}{\Omega^2}}.
\]
It follows that
\[
  \sum_{a \sim A}\#{\cal R}_a \ll \sum_{r\asymp R}\sum_{a\sim A}
  \charf{\norm*{\tilde f_a(r)}\ll \frac{H}{\Delta^2}\frac{1}{\Omega^2}}.
\]

To bound the inner sum, note that for all $a\sim A$, we have by \eqref{eq:std_exp}
and \eqref{eq:tilded_ests} that 
\[
  |\tilde f_a(r)|, A \bigg|\frac{\partial}{\partial a} \tilde f_a(r)\bigg|,
  A^2 \bigg|\frac{\partial^2}{{\partial a}^2} \tilde f_a(r)\bigg| \asymp F.
\]
Therefore, by Proposition \ref{prop:poisson_trivial} (the trivial bound applied following
Poisson summation after Fourier expanding), we have that
\[
  \sum_{a\sim A}\charf{\norm*{\tilde f_a(r)}\ll \frac{H}{A^2}}\ll \frac{H}{A}
  + \sqrt{F \frac{A^2}{H}} + \frac{A}{\sqrt{F A^2/H}}\ll \Delta
  \biggl( \frac{H}{\Delta^2}\Omega^{-1} + \sqrt{\frac{H}{\Delta^2}} G^{\frac{1}{2}}\Omega^{\frac{3}{2}}
  + \frac{1}{\sqrt{H}} \frac{1}{G^{\frac{1}{2}}\Omega^{\frac{1}{2}}}\biggr),
\]
and the desired result follows upon summing over $r\asymp R\asymp HG\Omega^3/\Delta$.

\subsection{Boosting the approximation of $d_a^*(r)$ and the proof of Proposition \ref{prop:upper_regime_est_nontrivial}}
At this point, in the critical case, we have essentially dealt with $\Delta \ggg H^{\frac 12}$.
Along with the results of \S\ref{sec:lower}, we are essentially left with the case $\Delta\asymp H^{\frac 12}$.
In this range, the error term in \eqref{eq:f_approx} in the critical case is essentially
$O(1)$, so $f_a^*(r)$ is determined up to $O(1)$ possibilities as it is an integer. These
possibilities may be expressed easily in terms of $\floor*{\tilde f_a}$.

The key idea is that once $f_a^*$ is determined, we may use the fact that
\[
  F_a(d_a^*(r)) = f + O\pfrc{1}{H\Delta^2} 
\]
to obtain a superior estimate for $d_a^*(r)$ than is provided by \eqref{eq:d_approx}.
As it turns out, where we are concerned, this estimate has error $\lll 1$, so the fact
that $d_a^*(r)$ is an integer is a nontrivial restriction that is not often satisfied.
We show this is so by noting that the approximation is essentially a bracket
polynomial (after restricting to a short interval), and thus comes from a polynomial
sequence on a nilmanifold whose equidistribution is well-understood in a sufficiently
quantitative form for us due to Green and Tao \cite{GT}.

In the general situation, by \eqref{eq:f_approx}, there exists an integer $|j|\ll 1 + H/a^2$
such that $f_a^*(r) = \floor*{\tilde f_a(r)} + j$. Let $\breve d_a(f) = F_a^{-1}(f)$. For $f\sim F$, we have 
\[
  F\breve d_a'(f), F^2\breve d_a''(f) \asymp H\Delta.
\]
Note that since
\[
  F_a(d_a^*(r)) = f_a^*(r)\biggl(1 + O\biggl(\frac{\Delta}{GH^3a}\biggr)\biggr),
\]
by \eqref{eq:std_exp}, we have 
\begin{equation}
  d_a^*(r) = F_a^{-1}(f_a^*(r))\biggl(1 + O\biggl(\frac{\Delta}{GH^3a}\biggr)\biggr)
  = \breve d_a(\floor*{\tilde f_a(r)} + j) + O\pfrc{\Delta^2}{H^2Ga}.
\end{equation}
Taylor expanding yields
\[
  \breve d_a(\tilde f_a(r) + j - \set*{\tilde f_a(r)})
  = \breve d_a(\tilde f_a(r) + j) - \breve d_a'(\tilde f_a(r) + j)\set*{\tilde f_a(r)}
  + O\biggl(\frac{\Delta^5}{H^3}\pfrc{\Delta}{a}^2\biggr).
\]
In particular, 
\[
  \norm*{\breve d_a(\tilde f_a(r) + j) - \breve d_a'(\tilde f_a(r) + j)\set*{\tilde f_a(r)}}
  \ll \frac{\Delta^5}{H^3}\pfrc{\Delta}{a}^2 + \frac{\Delta^2}{H^2Ga}.
\]
Write
\[
  g_j(r) = \breve d_a(\tilde f_a(r) + j) - \breve d_a'(\tilde f_a(r) + j)\set*{\tilde f_a(r)}, 
\]
By the preceding discussion,
\begin{equation}\label{eq:Ra_brack_bd}
  \#{\cal R}_a \le \sum_{|j|\ll 1 + \frac{H}{a^2}}\sum_{r\asymp R}
  \charf{\norm*{g_j(r)}\ll \frac{\Delta^5}{H^3}\pfrc{\Delta}{a}^2 + \frac{\Delta^2}{H^2Ga}}.
\end{equation}
Let $T_3 = H\Delta, T_2 = F, T_1 = {H\Delta}/F$. 
By the hypotheses of Proposition \ref{prop:upper_regime_est_nontrivial}, for $i\le 3$, we have
\[
  R^{i - \frac{1}{100}} \ll T_i \ll R^{i + \frac{1}{100}}.
\]
Therefore, the hypotheses of Proposition~\ref{prop:floor_fn_est} are satisfied and
there exists $\sigma > 0 $ such that $\set*{g_j(r)}_{r\asymp R}$ is 
$R^{-20\sigma}$-equidistributed mod 1. By \eqref{eq:Ra_brack_bd}, we obtain
\[
  \#{\cal R}_a\ll R\biggl(1 + \frac{H}{a^2}\biggr)\biggl(R^{-20\sigma} + \frac{\Delta^5}{H^3}\pfrc{\Delta}{a}^2 + \frac{\Delta^2}{H^2Ga}\biggr).
\]
Proposition \ref{prop:upper_regime_est_nontrivial} follows upon noting that the bounds we have
assumed in its hypotheses crudely imply that
\[
  R^{-20\sigma}  + \frac{\Delta^5}{H^3}\pfrc{\Delta}{a}^2 + \frac{\Delta^2}{H^2Ga}\ll X^{-\sigma'} (GU)^{O(1)}
\]
for some absolute $\sigma' > 0 $.

\section{Bracket polynomial equidistribution estimates}\label{sec:cq6v6y3nkr}
\label{sec:bracket_ests}

This section's main result is the following general estimate:
\begin{proposition}
  \label{prop:floor_fn_est}
  Take $\tau  = 1/{100}$ and $\sigma > 0$ sufficiently small. 
  Suppose that $T_1, T_2, T_3, N$ are sufficiently large parameters satisfying $N^{i - \tau}\ll T_i\ll N^{i + \tau}$.
  Furthermore, suppose that $g_1, g_2, g_3$ are smooth functions supported on values
  $\asymp N$ satisfying
  \[
    g_i^{(j)}(x)\asymp_j  \frac{T_i}{N^j},
  \]
  for all $x\asymp N$, $ j\ge 0, $ $ i\le 3$.
  Then, 
  \[
    g_3(n) + g_1(n)\set{g_2(n)}
  \]
  is $N^{-\sigma}$-equidistributed modulo $1$ for $n\sim N$.
\end{proposition}

After Taylor expanding, we shall show that Proposition \ref{prop:floor_fn_est} follows from the following:
\begin{proposition}\label{prop:bracket_poly}

  Take $\delta > 0$, $N\ge 1 $. There exist absolute constants $C_1, C_2 > 0$ such that the
  following holds.
  
  Let $f_1, f_2, f_3$ be polynomials of degrees $\le 3$. Suppose
  furthermore that the linear coefficient of $f_i(x)$ is $\alpha_i$.

  Then, if $f_3(n) + f_1(n)\set{f_2(n)}$ is not $\delta$-equidistributed
  modulo $1$ for $n\le N$, there exist $q_1, q_2, q_3\ll \delta^{-C_1}$ not all $0$ such that
  \[
    \norm{q_1\alpha_1 + q_2\alpha_2 + q_3\alpha_3}\ll \frac{\delta^{-C_2}}{N}.
  \]
\end{proposition}

Before describing the proof of Proposition \ref{prop:bracket_poly}, we will
show Proposition \ref{prop:floor_fn_est} assuming Proposition \ref{prop:bracket_poly}. 
\begin{proof}[Proof of Proposition \ref{prop:floor_fn_est} assuming Proposition \ref{prop:bracket_poly}]

  Write $g(n) = g_3(n) + g_1(n)\set{g_2(n)}$. Take $\delta = N^{-\sigma} $ and fix $\kappa = 1/10 $. 
  It is sufficient to show that there are at most $O(\delta N)$-many $n_0\sim N$ for which
  $g(n_0 + h)$ is not $\delta$-equidistributed for $h\le N^{\kappa}$.

  Consider some choice of $n_0\sim N, h\le N^{\kappa}$. Taylor expanding, we have
  \begin{align*}
    g_3(n_0 + h) &=  g_3(n_0) + g_3'(n_0)h + \frac{1}{2}g_3''(n_0)h^2 + \frac{1}{6} g_3'''(n_0) h^3
                   + O(N^{4\kappa + \tau - 1})\\
                 &:=  \tilde g_3(h; n_0) + O(N^{4\kappa + \tau - 1}).
  \end{align*}
  Similarly, we have that $g_i(n_0 + h) = \tilde g_i(h;n_0) + O(N^{4\kappa + \tau - 2})$ for $i\le 2$ with
  \[
    \tilde g_1(h; n_0) = g_1(n_0) + g_1'(n_0) h + \frac{1}{2} g_1''(n_0) h^2,\qquad \tilde g_2(h;n_0) = g_2(n_0) + g_2'(n_0)h
    + \frac{1}{2}g_2''(n_0)h^2 + \frac{1}{6}g_2'''(n_0)h^3.
  \]
  Let $\tilde g(h; n_0) = \tilde g_3(h; n_0) + \tilde g_1( h; n_0)\set{\tilde g_2(h; n_0)}$. Put $\varepsilon=N^{4\kappa+\tau-2}$.
  Outside $\|g_2(n_0+h)\|\ll N^{-\kappa}$, the fractional parts of $\tilde g_2(h; n_0)$ and $g_2(n_0 + h) $ differ by
  $O(\varepsilon)$. With the standard exponent pair $(1/6,1/6)$ (see \cite[\S8.4]{IK} for our exponent pair
  convention), $\#\{n\asymp N:\|g_2(n)\|\ll N^{-\kappa}\}\ll N^{11/12}$.
  In particular, when $\sigma<1/24$, for all but $O(\delta N)$-many $n_0\sim N $, 
  \begin{equation*}
    g(n_0+h)=\widetilde g(h;n_0)+O(N^{4\kappa+2\tau-1})
    =\widetilde g(h;n_0)+O(\delta)
  \end{equation*}
  for all but $O(\delta N^\kappa) $-many $h\le N^\kappa $.

  It remains to show that there are at most $O(\delta N)$-many $n_0\sim N$ for which
  $\tilde g(h;n_0) $ is not $O(\delta)$-equidistributed for $h\le N^{\kappa}$. Let $\mc N^\#$ be the set of such $n_0$.
  Suppose for the sake of contradiction that $\#\mathcal{N}^\# \ge \delta N$.
  By Proposition \ref{prop:bracket_poly}, if $n_0\in\mc N^\#$, there exist integers $q_1, q_2, q_3$ with
  $|q_1|, |q_2|, |q_3|\ll \delta^{-O(1)} $ and
  \begin{equation}\label{eq:rat_approx}
    \norm{g_1'(n_0) q_1 + g_2'(n_0) q_2 + g_3'(n_0) q_3}\ll \delta^{-O(1)}N^{-\kappa}.
  \end{equation}
  Let $\mc N(q_1, q_2, q_3)$ be the set of $n_0\in\mc N^\#$ satisfying \eqref{eq:rat_approx}.
  By the pigeonhole principle, there exists a choice of $|q_1|, |q_2|, |q_3|\ll\delta^{-O(1)}$ such that
  $\mc N(q_1, q_2, q_3)\gg \delta^{O(1)}N$. If $q_2 = q_3 = 0$, this is ruled out by Lemma \ref{lem:crude_frac_bound}.

  Otherwise, note that with $i\in\set{2, 3}$ maximal so that $q_i\ne 0$, we have
  $g(y) = g_1'(y) q_1 + g_2'(y) q_2 + g_3'(y) q_3$ satisfies $g^{(j)}(y)\asymp_j T/N^j $
  for all $j\ge 0$ for some $N^{\frac 12}\ll T\ll N^{2 + \frac 13}$ if $\sigma$ is sufficiently small.
  As before, the exponent pair $(1/6, 1/6)$ yields that $\norm{g(n_0)}\ll \delta^{-O(1)}N^{-\kappa}$
  for at most $O((\delta^{-O(1)}N^{-\kappa} + N^{-\frac{1}{12}})N)$-many $n_0\sim N$. For $\sigma$ sufficiently
  small, this implies the desired result.

\end{proof}
Proposition~\ref{prop:bracket_poly} will follow from 
Green and Tao's quantification of Leibman's theorem, stated below. We refer the reader to Green and Tao's paper \cite{GT}
for definitions regarding nilmanifolds and filtrations.
\begin{theorem}[{\cite[Theorem 2.9]{GT}}]\label{thm:gt_main}
  
  Suppose that $G/\Gamma$ is an $m$-dimensional nilmanifold, $G_\bullet$ a filtration of degree $d$, and $\mc X$ a $Q$-rational Mal'cev basis.
  Then, for $g\in\mathrm{poly}(\ZZ, G_\bullet), 0 < \delta < \frac 12$,
  we have at least one of the following:
  \begin{enumerate}
  \item
    $(g(n))_{n\le N}$ is $\delta$-equidistributed in $G/\Gamma$.
  \item
    There exists a non-trivial horizontal character $\eta : G\to\RR/\ZZ$ with $|\eta|\ll \delta^{-O(1)}$ such that
    \[
      \norm{\eta\circ g}_{C^\infty[N]}\ll \pfrc{Q}{\delta}^{O(1)}.
    \]
  \end{enumerate}
\end{theorem}

\begin{proof}[Proof of Proposition \ref{prop:bracket_poly}]\footnote{The initial setup here owes much to a related answer by FKR to the
    following MathOverflow question: \url{https://mathoverflow.net/questions/452447/}.}
  Take $H$ the Heisenberg group, and let $\Gamma_H$ consist of elements
  of $H$ with integer entries.
  The nilmanifold we start with is $G/\Gamma$ with 
  $G = \RR\times H$, $\Gamma = \ZZ\times\Gamma_H$.
  Then, the Lie algebra of $G$, $\mf g$, is given by
  \[
    \RR\times\set{
      \begin{pmatrix}
        0 & \RR & \RR \\
        0 & 0   & \RR \\
        0 & 0   &  0
      \end{pmatrix}}.
  \]
  As a Mal'cev basis, take
  \begin{align*}
    X_1 &= \left(1, 
          \begin{pmatrix}
            0 & 0   &  0 \\
            0 & 0   &  0 \\
            0 & 0   &  0
          \end{pmatrix}\right),
          X_2 = \left(0, \begin{pmatrix}
            0 & 1   &  0 \\
            0 & 0   &  0 \\
            0 & 0   &  0
          \end{pmatrix}\right),\\
    X_3 &= \left(0, \begin{pmatrix}
      0 & 0   &  0  \\
      0 & 0   &  1  \\
      0 & 0   &  0
    \end{pmatrix}\right),
          X_4 = \left(0, \begin{pmatrix}
            0 & 0   &  1 \\
            0 & 0   &  0 \\
            0 & 0   &  0
          \end{pmatrix}\right).
  \end{align*}
  Let $\phi(t_1, t_2, t_3, t_4) = \exp(t_1X_1)\exp(t_2 X_2)\exp(t_3X_3)\exp(t_4X_4)$, so that
  \[
    \phi(t_1, t_2, t_3, t_4) = \left(t_1, \begin{pmatrix}
        1 & t_2 & t_2t_3 + t_4 \\
        0 & 1   & t_3 \\
        0 & 0   &  1
      \end{pmatrix}\right).
  \]
  As our polynomial sequence, we shall take $g(n) = \phi(f_3(n), f_1(n), f_2(n), 0)$.
  It can be checked this is a polynomial sequence under
  the filtration $G_\bullet$ (see e.g. \cite[\S6]{GT}) given by
  \begin{equation}
    G_0 = G_1 = G_2 = G_3 = G,\qquad
    G_4 = G_5 = G_6 = [G, G],\qquad
    G_7 = \{e\}.
  \end{equation}

  Now, consider the map $F : G/\Gamma\to\RR/\ZZ$ given by
  \[
    F(\phi(t_1, t_2, t_3, t_4)\Gamma) = t_1 + (t_4 + t_2t_3) - t_2\floor{t_3}
    = t_1 + t_4 + t_2 \{t_3\}.
  \]
  That this is well-defined follows from the well-known map into the fundamental domain of
  $H/\Gamma_H$ consisting of elements of $H$ with entries in $[0, 1)$ given by
  \begin{equation*}
    \begin{pmatrix}
      1 & x & z \\
        & 1 & y \\
        &   & 1
    \end{pmatrix}\mapsto
    \begin{pmatrix}
      1 & \set{x} & \set{z - x\floor{y}} \\
        & 1 & \set{y} \\
        &   & 1
    \end{pmatrix}.
  \end{equation*}
  
  Strictly speaking, $F$ is not globally Lipschitz.
  The following lemma will help remedy this.
  \begin{lemma}
    If $(x_n)_{n\le N} $ is $\varepsilon $-equidistributed in $G/\Gamma $, then $(F(x_n))_{n\le N} $ is
    $O(\sqrt{\varepsilon}) $ equidistributed in $\mathbb{R}/\mathbb{Z} $.
  \end{lemma}
  \begin{proof}
    The third horizontal coordinate induces 
    \[
      \theta_3:G/\Gamma\longrightarrow \mathbb R/\mathbb Z,
      \qquad
      \theta_3(\phi(t_1,t_2,t_3,t_4)\Gamma)=t_3 + \mathbb{Z},
    \]
    which is well-defined.
    We are concerned with jumps at $\Sigma:=\theta_3^{-1}(0)$. Write $m_{G/\Gamma}$ and $m_{\mathbb{R}/\mathbb{Z}}$ for the normalized Haar
    probability measures on $G/\Gamma$ and $\mathbb{R}/\mathbb{Z}$, respectively.
    For $0<\rho<1/10$, put $\Sigma_\rho
    :=\{x\in G/\Gamma:\|\theta_3(x)\| <\rho\}$.
    Since $(\theta_3)_*m_{G/\Gamma}=m_{\mathbb{R}/\mathbb{Z}}$, we have $m_{G/\Gamma}(\Sigma_\rho)=2\rho$.

    Choose Lipschitz $\chi_\rho,\beta_\rho:G/\Gamma\to[0,1]$ such that
    $\chi_\rho=0$ on $\Sigma_\rho$, $\chi_\rho=1$ outside $\Sigma_{2\rho}$, and
    \[
      \mathbf 1_{\Sigma_{2\rho}}\le \beta_\rho
      \le \mathbf 1_{\Sigma_{3\rho}},\qquad
      \|\chi_\rho\|_{\mathrm{Lip}}
      +\|\beta_\rho\|_{\mathrm{Lip}}\ll\rho^{-1}.
    \]
    The equidistribution of $(x_n) $ implies 
    \[
      \frac{1}{N} \sum_{n\le N } \mathbf 1_{\Sigma_{2\rho}}(x_n)
      \le \frac{1}{N}\sum_{n\le N }\beta_\rho(x_n)
      \ll \rho+\frac{\varepsilon}{\rho}.
    \]
    If $\psi:\mathbb R/ \mathbb{Z} \to\mathbb C$ satisfies $\|\psi\|_{\mathrm{Lip}}\le1$, then $\Phi_\rho(x):=\chi_\rho(x)\psi(F(x))$ is
    globally Lipschitz and $\|\Phi_\rho\|_{\mathrm{Lip}}\ll\rho^{-1}$. Consequently,
    \begin{equation}
      \left|
        \mathbb E_{n\le N}\psi(F(x_n))
        -\int_{G/\Gamma}\psi(F(x))\,dm_{G/\Gamma}(x)
      \right|\le
    \left|
      \mathbb E_{n\le N}\Phi_\rho(x_n)
      -\int_{G/\Gamma}\Phi_\rho\,dm_{G/\Gamma}
    \right| + \rho + \frac{\varepsilon}{\rho}
      \ll \rho+\frac{\varepsilon}{\rho}.
    \end{equation}
    Finally, $F_*m_{G/\Gamma}=m_{\mathbb R/\mathbb{Z}}$ because the Haar-uniform $t_1$-coordinate occurs additively in $F$.
    Taking $\rho=\sqrt{\varepsilon}$, we conclude that $(F(x_n))_{n\le N}$ is $O(\sqrt{\varepsilon})$-equidistributed in $\mathbb R/\mathbb{Z}$.
  \end{proof}

  We have $F(g(n)\Gamma)=f_3(n)+f_1(n)\{f_2(n)\}$. The preceding lemma implies that there exists $0 < c \ll 1 $ such that 
  $c\delta^2$-equidistribution of $(g(n)\Gamma)_{n\le N}$ implies $\delta$-equidistribution of $(F(g(n)\Gamma))_{n\le N}$.
  Consequently, under the hypothesis of Proposition~\ref{prop:bracket_poly}, the orbit $(g(n)\Gamma)_{n\le N}$ is not $c\delta^2$-equidistributed
   in $G/\Gamma $.

  Applying Theorem~\ref{thm:gt_main} with parameter $\delta_0 = c\delta^2$, there exists a nontrivial horizontal character
  $\eta:G\to\RR/\ZZ$ such that
  \[
    |\eta|\ll \delta_0^{-O(1)}\ll\delta^{-O(1)},
    \qquad
    \norm{\eta\circ g}_{C^\infty[N]}
    \ll\delta_0^{-O(1)}\ll\delta^{-O(1)}.
  \]
  %

  Take $\pi : G\to (G/\Gamma)_{ab} = G/[G, G]\Gamma$ the natural projection.
  Also, $(G/\Gamma)_{ab}\cong (\RR/\ZZ)^3$, and in coordinates we have
  $\pi(\phi(t_1, t_2, t_3, t_4)) = (t_1, t_2, t_3)$, so there exist $q_1, q_2, q_3\in\ZZ$, not all zero, such that
  $\eta(\phi(t_1, t_2, t_3, t_4)) = q_3 t_1 + q_1t_2 + q_2t_3$. The bound on the modulus of
  $\eta$ implies that $|q_1|, |q_2|, |q_3|\ll\delta^{-O(1)}$. Also, note that
  $\eta\circ g(n) = q_1f_1(n) + q_2f_2(n) + q_3f_3(n)$, so the condition $\norm{\eta\circ g}_{C^\infty[N]}\ll\delta^{-O(1)}$
  implies that $\norm{q_1\alpha_1 + q_2\alpha_2 + q_3\alpha_3}\ll\delta^{-O(1)}/N$, as desired.
\end{proof}

\section{Acknowledgements}
I thank my advisor, Peter Sarnak, for discussions and encouragement.
I would also like to thank Tsz Ho Chan, Ben Green, Joni Ter\"av\"ainen, and Ognian Trifonov
for various comments on earlier versions.

\appendix
\section{Table of symbols}
For the ease of the reader, we include Table~\ref{tab:cq5rbqwg57} (appearing at the end of the paper)
listing a selection of the many symbols used through the paper, with a focus on those that
span multiple sections, as well as a reference to the first equation
or section/subsection in which they are first defined (along with, in some selected cases,
further constraints that are important). The symbols are ordered by their first occurrence.


\setcounter{table}{0} \renewcommand{\thetable}{A}
\begin{table}[h]  \label{tab:cq5rbqwg57}
  \begin{tabular}{c c c}
    Symbol & Definition/salient constraints & First occurrence \\ \hline
    $H$ & $H\asymp X^{\frac{1}{5} - \eta}$ & \S\ref{sec:cq5rbolqen} \\ \hline
    $G$ & $X/H^5\asymp X^{5\eta}$ & \S\ref{sec:cq5rbolqen} \\ \hline
    $U$ & $U = X^{\eta}$ & \S\ref{sec:cq5rbolqen} \\ \hline
    $D_-$ & $H/U$ & \S\ref{subsec:reduction}\\ \hline
    $D_+$ & $2X^{1/2}$ & \S\ref{subsec:reduction} \\ \hline
    $\mathcal{D}_S$ & $\set{d\in S : \text{there exists } m\in \mathbb{Z} \text{ such that } md^2\in [X, X + H]}$ & \eqref{eq:cq5rbpe0xy}\\ \hline
    $D$ & $D_-\le D\le D_+$ & \S\ref{subsec:reduction} \\ \hline
    $\Delta$ & $D/H$ & \S\ref{subsec:reduction}\\  \hline
    $\mathcal{D}$ & $\mathcal{D}_{[D, 2D]}$ & \S\ref{subsec:reduction}\\  \hline
    $\mathcal{D}_a$ & $\set{d : d, d + a\in \mathcal{D}\text{ such that } (d, d + a)\cap \mathcal{D} = \emptyset}$& \eqref{eq:cq5rbpwe1z}\\ \hline
    $A$ & $\Delta^{\frac{4}{3}} (H^4/X)^{\frac{1}{3}}\ll A\ll \Delta U $; $ \Delta (GU)^{-O(1)}\ll A$ & \eqref{eq:pigeonholed}; \S\ref{sec:final_reductions}\\ \hline
    $\Omega$ & $A/\Delta$ & \S\ref{sec:final_reductions} \\ \hline
    $F_a(d)$ & $\dfrac{X}{d^2} - \dfrac{X}{(d + a)^2}$ & \eqref{eq:cq5rbqjd9g}\\ \hline
    $R_a(d)$ & $-(2d - a)\dfrac{X}{d^2} + (2d + 3a)\dfrac{X}{(d + a)^2} $ & \eqref{eq:cq5rbql9yk}\\ \hline
    $R$ & $ \dfrac{XA^3}{\Delta^4H^4}\asymp \dfrac{H}{\Delta} G\Omega^3 $ & \eqref{eq:11}\\\hline 
    $\tilde d_a$ & $R_a^{-1}$ & \S\ref{sec:cq5rbqt63z} \\
           & $\tilde d_a(\rho)\asymp D$ for $\rho \asymp R$ & \\ \hline
    $\mathcal{R}_a$ & \eqref{eq:D_aR_a_size_bound}, \eqref{eq:d_approx} & Proposition~\ref{prop:approx_param}\\ \hline
    $d_a^*(r)$ & \eqref{eq:4}, \eqref{eq:d_approx}& Proposition~\ref{prop:approx_param}\\ \hline
    $r_a^*(d)$ & $\mathbb{Z}\ni r_a^*(d)\asymp R$, $R_a(d) = r_a^*(d) + O(1/\Delta)$ & \eqref{eq:3}\\ \hline
    $W$ & $\asymp GU^5$ & Proposition~\ref{prop:Ra_bound_lower} \\  \hline
    $\ell_1, \ell_2$ & $\ell_1 < \ell_2\ll W$ & \S\ref{sec:cq5rctx73a}\\  \hline
    $b_0, v$ & $d + \ell_1b_0 = d_a^*(r + \ell_1), d + \ell_2 b_0 + v = d_a^*(r + \ell_2)$ & \eqref{eq:10}\\
           & $|b_0| \asymp \Delta^2 (GU)^{O(1)}, v\ll \Delta (GU)^{O(1)}$ & \\ \hline
    $B$ & $\dfrac{D}{R}\asymp \Delta^2 \dfrac{1}{G\Omega^3}$, $|b_0| \asymp B$ & \eqref{eq:1}\\ \hline
    $ S_{a, b}(d)$ & $-(b - a)\dfrac{X}{d^2} + (b + a)\dfrac{X}{(d + a)^2} - (b + a)\dfrac{X}{(d + b)^2} + (b - a)\dfrac{X}{(d + a + b)^2}$ & \eqref{eq:2}\\
           & $\asymp \Delta^2 (GU)^{O(1)}$ & \\ \hline
    $\hat S_{a, b}(d)$ & $S_{a, b}(d) - (R_a(d) - R_a(d + b))$ & \eqref{eq:6}\\\hline
    $F_{a, b}(d)$ & $ F_a(d) - F_a(d + b)$ & \eqref{eq:7}\\
           & $\asymp \dfrac{H}{\Delta}(GU)^{O(1)}$ & \\ \hline
    $\Qoppa$ & $  \ell_1 F_{a, \ell_2b_0 + v}(d) - \ell_2 F_{a, \ell_1b_0}(d)$ & \eqref{eq:9}\\
           & $\ll \biggl( \dfrac{H}{\Delta^3}|v|+ 1 \biggr) (GU)^{O(1)}$ &\\  \hline
    $ \Upsilon$ & $  \ell_2^2(\ell_2 - \ell_1)^2\hat S_{a, b_0\ell_1}(d) - \ell_1^2(\ell_2 - \ell_1)^2\hat S_{a, b_0\ell_2 + v}(d) $ & \eqref{eq:8}\\
           & $ + \ell_1^2\ell_2^2\hat S_{a, b_0(\ell_2 - \ell_1) + v}(d + \ell_1b_0) $ & \\ \hline
    $p_1(v)$ & $3\ell_1^3\ell_2(\ell_2 - \ell_1)b_0v^2 + \ell_1^3(2\ell_2 - \ell_1)v^3$ & \eqref{multline:cq5rr17ch5}\\
    $p_2(v)$ & $-5\ell_1^3\ell_2^2(\ell_2 - \ell_1)^2b_0^3v - 15\ell_1^3\ell_2^2(\ell_2 - \ell_1)b_0^2v^2$ & \eqref{multline:cq5rr2avlb}\\ 
           & $
             - 5\ell_1^3\ell_2(3\ell_2 - 2\ell_1)b_0v^3 - \frac{5}{2}\ell_1^3(2\ell_2 - \ell_1)v^4$ & \\ \hline

    $\Sampi(r)$ & $ \dfrac{X}{\tilde d_a(r)^5}\biggl(-4a + 10\dfrac{a^2}{\tilde d_a(r)}\biggr)\biggl(\tilde p_1(v, r) + \dfrac{\tilde p_2(v, r)}{\tilde d_a(r)}\biggr)$ & \eqref{eq:cq6v987zct}\\ \hline
  \end{tabular}
  \caption{}
\end{table}

\bibliography{refs}{} \bibliographystyle{alpha}

@book {IK,
    AUTHOR = {Iwaniec, Henryk and Kowalski, Emmanuel},
     TITLE = {Analytic number theory},
    SERIES = {American Mathematical Society Colloquium Publications},
    VOLUME = {53},
 PUBLISHER = {American Mathematical Society, Providence, RI},
      YEAR = {2004},
     PAGES = {xii+615},
      ISBN = {0-8218-3633-1},
   MRCLASS = {11-02 (11Fxx 11Lxx 11Mxx 11Nxx)},
  MRNUMBER = {2061214},
MRREVIEWER = {K.\ Soundararajan},
       DOI = {10.1090/coll/053},
       URL = {https://doi.org/10.1090/coll/053},
  display-string = {(2004) Iwaniec, Henryk and Kowalski, Emmanuel - Analytic number theory},
}

@article {D,
    AUTHOR = {Deshouillers, Jean-Marc},
     TITLE = {Sur la r\'epartition des nombres {$[n\sp{c}]$} dans les
              progressions arithm\'etiques},
   JOURNAL = {C. R. Acad. Sci. Paris S\'er. A-B},
  FJOURNAL = {Comptes Rendus Hebdomadaires des S\'eances de l'Acad\'emie des
              Sciences. S\'eries A et B},
    VOLUME = {277},
      YEAR = {1973},
     PAGES = {A647--A650},
      ISSN = {0151-0509},
   MRCLASS = {10H25 (10FXX 10H15)},
  MRNUMBER = {337834},
MRREVIEWER = {H.\ G.\ Meijer},
  display-string = {(1973) Deshouillers, Jean-Marc - Sur la repartition des nombres $[n\sp{c}]$ dans les
              progressions arithmetiques},
}

@article {GT,
    AUTHOR = {Green, Ben and Tao, Terence},
     TITLE = {The quantitative behaviour of polynomial orbits on
              nilmanifolds},
   JOURNAL = {Ann. of Math. (2)},
  FJOURNAL = {Annals of Mathematics. Second Series},
    VOLUME = {175},
      YEAR = {2012},
    NUMBER = {2},
     PAGES = {465--540},
      ISSN = {0003-486X,1939-8980},
   MRCLASS = {37A15},
  MRNUMBER = {2877065},
MRREVIEWER = {Tamar\ Ziegler},
       DOI = {10.4007/annals.2012.175.2.2},
       URL = {https://doi.org/10.4007/annals.2012.175.2.2},
  display-string = {(2012) Green, Ben and Tao, Terence - The quantitative behaviour of polynomial orbits on
              nilmanifolds},
}

@article {M,
    AUTHOR = {Montgomery, H. L.},
     TITLE = {A note on the large sieve},
   JOURNAL = {J. London Math. Soc.},
  FJOURNAL = {The Journal of the London Mathematical Society},
    VOLUME = {43},
      YEAR = {1968},
     PAGES = {93--98},
      ISSN = {0024-6107,1469-7750},
   MRCLASS = {10.64},
  MRNUMBER = {224585},
MRREVIEWER = {W.\ G. H. Schaal},
       DOI = {10.1112/jlms/s1-43.1.93},
       URL = {https://doi.org/10.1112/jlms/s1-43.1.93},
  display-string = {(1968) Montgomery, H. L. - A note on the large sieve},
}

@article {V,
    AUTHOR = {Vaughan, R. C.},
     TITLE = {Some applications of {M}ontgomery's sieve},
   JOURNAL = {J. Number Theory},
  FJOURNAL = {Journal of Number Theory},
    VOLUME = {5},
      YEAR = {1973},
     PAGES = {64--79},
      ISSN = {0022-314X,1096-1658},
   MRCLASS = {10H30},
  MRNUMBER = {342476},
MRREVIEWER = {H.\ L.\ Montgomery},
       DOI = {10.1016/0022-314X(73)90059-0},
       URL = {https://doi.org/10.1016/0022-314X(73)90059-0},
  display-string = {(1973) Vaughan, R. C. - Some applications of Montgomery's sieve},
}

@book {H,
    AUTHOR = {Huxley, M. N.},
     TITLE = {Area, lattice points, and exponential sums},
    SERIES = {London Mathematical Society Monographs. New Series},
    VOLUME = {13},
      NOTE = {Oxford Science Publications},
 PUBLISHER = {The Clarendon Press, Oxford University Press, New York},
      YEAR = {1996},
     PAGES = {xii+494},
      ISBN = {0-19-853466-3},
   MRCLASS = {11L07 (11J54 11P21)},
  MRNUMBER = {1420620},
MRREVIEWER = {R.\ C.\ Baker},
  display-string = {(1996) Huxley, M. N. - Area, lattice points, and exponential sums},
}

@article {B,
    AUTHOR = {Bourgain, J.},
     TITLE = {Decoupling, exponential sums and the {R}iemann zeta function},
   JOURNAL = {J. Amer. Math. Soc.},
  FJOURNAL = {Journal of the American Mathematical Society},
    VOLUME = {30},
      YEAR = {2017},
    NUMBER = {1},
     PAGES = {205--224},
      ISSN = {0894-0347,1088-6834},
   MRCLASS = {11M06 (11L07)},
  MRNUMBER = {3556291},
MRREVIEWER = {Kostadinka\ Lapkova},
       DOI = {10.1090/jams/860},
       URL = {https://doi.org/10.1090/jams/860},
  display-string = {(2017) Bourgain, J. - Decoupling, exponential sums and the Riemann zeta function},
}

@preamble{
   "\def\cprime{$'$} "
}

@article {MR4843,
    AUTHOR = {Fogels, E.},
     TITLE = {On average values of arithmetic functions},
   JOURNAL = {Proc. Cambridge Philos. Soc.},
  FJOURNAL = {Proceedings of the Cambridge Philosophical Society},
    VOLUME = {37},
      YEAR = {1941},
     PAGES = {358--372},
      ISSN = {0008-1981},
   MRCLASS = {10.0X},
  MRNUMBER = {4843},
MRREVIEWER = {P.\ Hartman},
  display-string = {(1941) Fogels, E. - On average values of arithmetic functions},
}

@article {MR43119,
    AUTHOR = {Roth, K. F.},
     TITLE = {On the gaps between squarefree numbers},
   JOURNAL = {J. London Math. Soc.},
  FJOURNAL = {The Journal of the London Mathematical Society},
    VOLUME = {26},
      YEAR = {1951},
     PAGES = {263--268},
      ISSN = {0024-6107,1469-7750},
   MRCLASS = {10.0X},
  MRNUMBER = {43119},
MRREVIEWER = {L.\ Carlitz},
       DOI = {10.1112/jlms/s1-26.4.263},
       URL = {https://doi.org/10.1112/jlms/s1-26.4.263},
  display-string = {(1951) Roth, K. F. - On the gaps between squarefree numbers},
}

@article {MR2505444,
    AUTHOR = {Bombieri, Enrico and Bourgain, Jean},
     TITLE = {On {K}ahane's ultraflat polynomials},
   JOURNAL = {J. Eur. Math. Soc. (JEMS)},
  FJOURNAL = {Journal of the European Mathematical Society (JEMS)},
    VOLUME = {11},
      YEAR = {2009},
    NUMBER = {3},
     PAGES = {627--703},
      ISSN = {1435-9855,1435-9863},
   MRCLASS = {42A05 (11C08 11L03 42A16)},
  MRNUMBER = {2505444},
MRREVIEWER = {Serge\u i\ V.\ Konyagin},
       DOI = {10.4171/jems/163},
       URL = {https://doi.org/10.4171/jems/163},
  display-string = {(2009) Bombieri, Enrico and Bourgain, Jean - On Kahane's ultraflat polynomials},
}

@article {MR1057318,
    AUTHOR = {Filaseta, Michael},
     TITLE = {Short interval results for squarefree numbers},
   JOURNAL = {J. Number Theory},
  FJOURNAL = {Journal of Number Theory},
    VOLUME = {35},
      YEAR = {1990},
    NUMBER = {2},
     PAGES = {128--149},
      ISSN = {0022-314X,1096-1658},
   MRCLASS = {11N25},
  MRNUMBER = {1057318},
MRREVIEWER = {S.\ W.\ Graham},
       DOI = {10.1016/0022-314X(90)90108-4},
       URL = {https://doi.org/10.1016/0022-314X(90)90108-4},
  display-string = {(1990) Filaseta, Michael - Short interval results for squarefree numbers},
}

@article {MR1171549,
    AUTHOR = {Filaseta, Michael and Trifonov, Ognian},
     TITLE = {On gaps between squarefree numbers. {II}},
   JOURNAL = {J. London Math. Soc. (2)},
  FJOURNAL = {Journal of the London Mathematical Society. Second Series},
    VOLUME = {45},
      YEAR = {1992},
    NUMBER = {2},
     PAGES = {215--221},
      ISSN = {0024-6107,1469-7750},
   MRCLASS = {11N25 (11B05)},
  MRNUMBER = {1171549},
MRREVIEWER = {S.\ W.\ Graham},
       DOI = {10.1112/jlms/s2-45.2.215},
       URL = {https://doi.org/10.1112/jlms/s2-45.2.215},
  display-string = {(1992) Filaseta, Michael and Trifonov, Ognian - On gaps between squarefree numbers. II},
}

@article {MR1397690,
    AUTHOR = {Filaseta, Michael and Trifonov, Ognian},
     TITLE = {The distribution of fractional parts with applications to gap
              results in number theory},
   JOURNAL = {Proc. London Math. Soc. (3)},
  FJOURNAL = {Proceedings of the London Mathematical Society. Third Series},
    VOLUME = {73},
      YEAR = {1996},
    NUMBER = {2},
     PAGES = {241--278},
      ISSN = {0024-6115,1460-244X},
   MRCLASS = {11J54 (11N25 11P21)},
  MRNUMBER = {1397690},
MRREVIEWER = {G.\ Kolesnik},
       DOI = {10.1112/plms/s3-73.2.241},
       URL = {https://doi.org/10.1112/plms/s3-73.2.241},
  display-string = {(1996) Filaseta, Michael and Trifonov, Ognian - The distribution of fractional parts with applications to gap
              results in number theory},
}

@article {MR967330,
    AUTHOR = {Graham, S. W. and Kolesnik, G.},
     TITLE = {On the difference between consecutive squarefree integers},
   JOURNAL = {Acta Arith.},
  FJOURNAL = {Polska Akademia Nauk. Instytut Matematyczny. Acta Arithmetica},
    VOLUME = {49},
      YEAR = {1988},
    NUMBER = {5},
     PAGES = {435--447},
      ISSN = {0065-1036},
   MRCLASS = {11L40 (11N37)},
  MRNUMBER = {967330},
MRREVIEWER = {H.-E.\ Richert},
       DOI = {10.4064/aa-49-5-435-447},
       URL = {https://doi.org/10.4064/aa-49-5-435-447},
  display-string = {(1988) Graham, S. W. and Kolesnik, G. - On the difference between consecutive squarefree integers},
}

@article {MR1654759,
    AUTHOR = {Granville, Andrew},
     TITLE = {{$ABC$} allows us to count squarefrees},
   JOURNAL = {Internat. Math. Res. Notices},
  FJOURNAL = {International Mathematics Research Notices},
      YEAR = {1998},
    NUMBER = {19},
     PAGES = {991--1009},
      ISSN = {1073-7928,1687-0247},
   MRCLASS = {11N32 (11D75)},
  MRNUMBER = {1654759},
MRREVIEWER = {G.\ Greaves},
       DOI = {10.1155/S1073792898000592},
       URL = {https://doi.org/10.1155/S1073792898000592},
  display-string = {(1998) Granville, Andrew - $ABC$ allows us to count squarefrees},
}

@article {MR1318755,
    AUTHOR = {Huxley, M. N. and Sargos, P.},
     TITLE = {Points entiers au voisinage d'une courbe plane de classe
              {$C^n$}},
   JOURNAL = {Acta Arith.},
  FJOURNAL = {Acta Arithmetica},
    VOLUME = {69},
      YEAR = {1995},
    NUMBER = {4},
     PAGES = {359--366},
      ISSN = {0065-1036,1730-6264},
   MRCLASS = {11J54},
  MRNUMBER = {1318755},
MRREVIEWER = {D.\ R.\ Heath-Brown},
       DOI = {10.4064/aa-69-4-359-366},
       URL = {https://doi.org/10.4064/aa-69-4-359-366},
  display-string = {(1995) Huxley, M. N. and Sargos, P. - Points entiers au voisinage d'une courbe plane de classe
              $C^n$},
}

@article {MR531167,
    AUTHOR = {Nair, M.},
     TITLE = {Power free values of polynomials. {II}},
   JOURNAL = {Proc. London Math. Soc. (3)},
  FJOURNAL = {Proceedings of the London Mathematical Society. Third Series},
    VOLUME = {38},
      YEAR = {1979},
    NUMBER = {2},
     PAGES = {353--368},
      ISSN = {0024-6115,1460-244X},
   MRCLASS = {10H25},
  MRNUMBER = {531167},
MRREVIEWER = {G.\ Greaves},
       DOI = {10.1112/plms/s3-38.2.353},
       URL = {https://doi.org/10.1112/plms/s3-38.2.353},
  display-string = {(1979) Nair, M. - Power free values of polynomials. II},
}

@article {MR72170,
    AUTHOR = {Rankin, R. A.},
     TITLE = {Van der {C}orput's method and the theory of exponent pairs},
   JOURNAL = {Quart. J. Math. Oxford Ser. (2)},
  FJOURNAL = {The Quarterly Journal of Mathematics. Oxford. Second Series},
    VOLUME = {6},
      YEAR = {1955},
     PAGES = {147--153},
      ISSN = {0033-5606,1464-3847},
   MRCLASS = {10.1X},
  MRNUMBER = {72170},
MRREVIEWER = {Paul\ T.\ Bateman},
       DOI = {10.1093/qmath/6.1.147},
       URL = {https://doi.org/10.1093/qmath/6.1.147},
  display-string = {(1955) Rankin, R. A. - Van der Corput's method and the theory of exponent pairs},
}

@article{10.1112/jlms/s1-29.1.16,
    author = {Richert, H.-E.},
    title = {On the Difference between Consecutive Squarefree Numbers},
    journal = {Journal of the London Mathematical Society},
    volume = {s1-29},
    number = {1},
    pages = {16-20},
    year = {1954},
    month = {01},
    issn = {0024-6107},
    doi = {10.1112/jlms/s1-29.1.16},
    url = {https://doi.org/10.1112/jlms/s1-29.1.16},
    eprint = {https://academic.oup.com/jlms/article-pdf/s1-29/1/16/2589819/s1-29-1-16.pdf},
  display-string = {(1954) Richert, H.-E. - On the Difference between Consecutive Squarefree Numbers},
}

@article{2021arXiv2110.09990,
  title                    = {New small gaps between squarefree numbers},
  author                   = {Tsz Ho Chan},
  year                     = {2021},
  month                    = {10},
  url                      = {http://arxiv.org/abs/2110.09990v2},
  note                     = {arXiv:2110.09990},
  eprinttype               = {arXiv}
}

@phdthesis{Sch,
  author = {P. G. Schmidt},
  title = {Absch\"atzungen bei unsymmetrischen Gitterpunktproblemen},
  school = {Georg--August--Universit\"at zu G\"ottingen},
  year = {1964},
  type = {Dissertation zur Erlangung des Doktorgrades der Mathematisch--Naturwissenschaftlichen Fakult\"at}
}

@article {MR985889,
    AUTHOR = {Trifonov, O.},
     TITLE = {On the squarefree problem},
   JOURNAL = {C. R. Acad. Bulgare Sci.},
  FJOURNAL = {Doklady Bolgarsko\u i\ Akademii Nauk. Comptes Rendus de
              l'Acad\'emie Bulgare des Sciences},
    VOLUME = {41},
      YEAR = {1988},
    NUMBER = {12},
     PAGES = {37--40},
      ISSN = {0366-8681},
   MRCLASS = {11N25},
  MRNUMBER = {985889},
MRREVIEWER = {S.\ W.\ Graham},
}

@article {MR1048051,
    AUTHOR = {Trifonov, Ognian},
     TITLE = {On the squarefree problem. {II}},
   JOURNAL = {Math. Balkanica (N.S.)},
  FJOURNAL = {Mathematica Balkanica. New Series},
    VOLUME = {3},
      YEAR = {1989},
    NUMBER = {3-4},
     PAGES = {284--295},
      ISSN = {0205-3217},
   MRCLASS = {11N25},
  MRNUMBER = {1048051},
MRREVIEWER = {S.\ W.\ Graham},
}

\end{document}